\documentclass[10pt,a4paper]{article}
\usepackage[latin1]{inputenc}
\usepackage{amsmath}
\usepackage{amsfonts}
\usepackage{amssymb}
\usepackage{amsthm}
\usepackage{fullpage}
\usepackage{epsfig}

\newcommand{\Ric}{\textnormal{Ric}}
\newcommand{\sca}{\textnormal{sc}}
\newcommand{\bi}{\textnormal{bi}}
\newcommand{\fl}{\textnormal{flag}}
\newcommand{\dist}{\textnormal{dist}\,}

\newcommand{\tr}{\textnormal{tr}}
\newcommand{\mc}{\mathcal}
\newcommand{\mb}{\mathbf}
\newcommand{\eps}{\varepsilon}
\newcommand{\la}{\langle}
\newcommand{\ra}{\rangle}
\newcommand{\id}{\textnormal{id}}
\newcommand{\we}{\wedge}
\newcommand{\ve}{\varepsilon}
\newcommand{\vf}{\varphi}

\newtheorem{lemma}{Lemma}[section]
\newtheorem{prop}[lemma]{Proposition}
\newtheorem{thm}[lemma]{Theorem}
\newtheorem{deflem}[lemma]{Definition, Lemma}
\newtheorem{defin}[lemma]{Definition}
\newtheorem{notation}[lemma]{Notation}
\newtheorem{remark}[lemma]{Remark}
\newtheorem{cor}[lemma]{Corollary}

\numberwithin{equation}{section}

\begin{document}
\title{Gluing Riemannian manifolds \\with curvature operators at least $\kappa$}
\author{Arthur Schlichting\footnote{The author was supported by the SFB TR 71}}
\date{October 10, 2012}
\maketitle
\begin{abstract}
\noindent
Let $(M_0,g_0)$ and $(M_1,g_1)$ be smooth Riemannian manifolds with smooth compact boundaries and Riemannian curvature operators $\geq\kappa$ (which means that 
all eigenvalues of the curvature operators are at least $\kappa$), and let $M$  be 
the Riemannian manifold resulting from gluing $M_0$ and $M_1$ along some isometry of their boundaries. The metrics $g_0$ and $g_1$ induce a 
continuous metric $g$ on $M$. If the sum of the second fundamental forms of 
the common boundary of $M_0$ and $M_1$ with respect to the inward normals is  positive semidefinite, then $g$ can be approximated by 
smooth metrics which have curvature operators almost $\geq\kappa$. An analogous result holds for manifolds
with with lower bounds on Ricci curvature, scalar curvature 
(in this case it suffices to assume that the sum of the mean curvatures of the boundary is nonnegative), bi-curvature, isotropic curvature, and flag curvature, respectively.
\end{abstract}

\section{Introduction and statement of results}
Gluing Alexandrov spaces of bounded curvature has been studied in a number of works, in particular by Reshetnyak \cite{re} (curvature bounded from above)
and Petrunin \cite{pet} (curvature bounded from below).
The case where the spaces being glued are smooth Riemannian manifolds of sectional curvature at least $\kappa$
was studied by Kosovski\u{\i} \cite{k}, where he shows that the resulting Alexandrov space has 
curvature at least $\kappa$ if and only if the sum of the second fundamental forms
of the initial manifolds on their common boundary is positive semidefinite. In this paper we shall examine a similar setup for smooth Riemannian
manifolds with smooth compact boundaries and curvature operators $\geq\kappa$. The method being used in \cite{k} can be applied
with some modifications.
\\
Let $(M,g)$ be a smooth Riemannian manifold with a smooth metric $g$, and let $\Lambda^2 (TM)\subset TM\otimes TM$ be the bundle of two-vectors  over $M$.
Given a point $p\in M$ and a basis $\{e_1,\dots, e_n\}$ of $T_pM$, $\Lambda^2 (T_pM)$ is generated by
\[\{e_i\we e_j=e_i\otimes e_j-e_j\otimes e_i  | 1\leq i<j\leq n\}\]
The metric $g$ induces an inner product $\mc I^g$ on $\Lambda^2(TM)$, defined by
\begin{eqnarray}\label{def I}
\mc I^g(e_i\we e_j,e_k\we e_l):=g_{ik}g_{jl}-g_{jk}g_{il}
\end{eqnarray}
where  $g_{ik}=g(e_i,e_k)$. Note that if the vectors $e_i$ are orthonormal with respect to $g$ then the two-vectors $e_i\we e_j$ are orthonormal with
respect to $\mc I^g$.
Let $R^g=(R^g_{ijkl})$ be the Riemannian curvature tensor of $g$ and $R^g_{ijkl}=R^g(e_i,e_j,e_k,e_l)$
(We choose the sign of $R^g_{ijkl}$ such that the sectional curvature of a two-plane spanned by some orthonormal $e_i,e_j$ is given by $R^g_{ijij}$.)
$R^g$ induces a symmetric bilinear form  $\mc R^g$ on $\Lambda^2(TM)$ via
\[\mc R^g(e_i\we e_j,e_k\we e_l)=R^g_{ijkl}\]
The Riemannian curvature operator on $\Lambda^2(TM)$, which we shall also denote by $\mc R^g$, is defined by the property
\[\mc I^g(\cdot,\mc R^g\cdot)=\mc R^g(\cdot,\cdot)\]
By $\mc R^g\geq \kappa\in \mathbb R$ (or $\mc R^g\geq\kappa \mc I^g$) we mean that all eigenvalues of $\mc R^g$ are at least $\kappa$, or equivalenly that
\[\mathcal R^g(\alpha,\alpha)\geq \kappa \mathcal I^g(\alpha,\alpha)\] holds for all $\alpha\in \Lambda^2(TM)$.
\\

Let us glue two Riemannian manifolds $(M_0,g_0)$ and $(M_1,g_1)$ along some isometry $\phi$ of their boundaries 
(which means we identify $p\in \partial M_0$ and $\phi(p)\in\partial M_1$), and define
a metric $g$ on the resulting manifold $M$  by $g|_{M_i}=g_i$, $i=0,1$. Due to the isometry of the boundaries, $g$ is continuous, but fails to be $C^2$-smooth in general.
In this case we can not speak of the Riemannian curvature operator of $g$ in the classical sense. In
\cite{k} Kosovski\u{\i} made use of the fact that nevertheless $M$ can be equipped with a length structure induced 
by $g$ and instead of bounded sectional curvature in the classical sense one has 
the notion of bounded curvature in the sense of Alexandrov (see, e.g., \cite{burago}). However,
there is no analogue of this notion for bounds on the Riemannian curvature operator. We introduce the following definition:

\begin{defin}\label{c0curv}
Let $M$ be a Riemannian manifold, equipped with a continuous metric $g$. We say that the Riemannian curvature operator of $g$ is at least $\kappa$
iff there exists a family of $C^\infty$-metrics $(g_{(\delta)})$ on $M$ which converge to $g$ uniformly on every compact subset
as $\delta$ tends to zero 
and
\[\mathcal R(g_{(\delta)}) \geq  \bigl(\kappa-\varepsilon(\delta)\bigr) \mathcal I(g_{(\delta)})\]
holds with $\varepsilon(\delta)\rightarrow 0$. 
\end{defin}

In view of the above definitions the main result of this paper is the following

\begin{thm}\label{main}
Let $M_0$ and $M_1$ be smooth Riemannian manifolds with (at least $C^2$-)smooth metrics $g_0$ and $g_1$ 
and smooth compact boundaries $\Gamma_0$ and $\Gamma_1$, respectively. 
Let $L_i$ be the second fundamental form of $\Gamma_i$ with respect to the inward normal $N_i$, $i=0,1$.
Suppose that there exists an isometry $\phi: \Gamma_0\rightarrow \Gamma_1$,
and let $M=M_0\cup_\phi M_1$ denote the
manifold obtained from gluing $M_0$ and $M_1$ along $\phi$. Then
$\Gamma:=\Gamma_0=_\phi \Gamma_1$ can be seen as a hypersurface of $M$ and we may define $L=L_0+L_1$ on $\Gamma$. 

Let $g$ be the continuous metric on $M$ induced by $g_0$ and $g_1$. Suppose that $\mathcal R(g_0)$ and $\mathcal R(g_1)$ are at least $\kappa$. 
If $L$ is positive semidefinite, then $\mathcal R(g)\geq \kappa$ in the sense of Definition \ref{c0curv}.

\end{thm}

Analogous results hold for manifolds with lower bounds on Ricci curvature, scalar curvature (in this case it suffices to require only that $\tr_gL\geq 0$ on $\Gamma$), 
bi-curvature (the sum of the two
smallest eigenvalues of the curvature operator), isotropic curvature and flag curvature, respectively. 
\\
\\
Plan of the proof of Theorem \ref{main}:
\\
\\
We proceed similarly to \cite{k}:
\begin{itemize}
\item In Section 2 we sum up auxiliary constructions. We introduce a smooth structure on $M$ relative to which $M_0$,
$M_1$ and their common boundary $\Gamma$ are smooth submanifolds. The metric $g$ on $M$ induced by 
$g_0$ and $g_1$ is continuous.

By modifying the metric $g_0$ near $\Gamma$ we construct a new metric $g_\delta$ on $M_0$. We then define a metric 
$g_{(\delta)}$ on $M$ by $g_{(\delta)}|_{M_0}=g_\delta$ and $g_{(\delta)}|_{M_1}=g_1$. The coefficients of this 
metric belong to the Sobolev class $W^{2,\infty}_{loc}$.

All of the constructions in this section have been adopted as it stands from \cite{k}, therefore we suppress the proofs
to the greatest extent. A more detailed discussion can be found in \cite{k}, \S\S~3-6.
\item
In Section 3 we compare the Riemannian curvature operators with respect to $g_\delta$ and $g_0$ on $M_0$. 
This section corresponds with \S~7 in \cite{k}.
\item
In Sections 4 and 5 we estimate the curvature operator of $g_\delta$, showing that 
$\mathcal R(g_\delta)\geq \kappa-\varepsilon(\delta)$ holds on $M_0$, which implies
$\mathcal R(g_{(\delta)})\geq \kappa-\varepsilon(\delta)$ a.e. on $M$ for the weakly defined curvature operator of 
the $W^{2,\infty}_{loc}$-metric $g_{(\delta)}$.
\item
In Section 6 we  mollify $g_{(\delta)}$
and construct a family of smooth metrics as required in Definition \ref{c0curv}. 
\end{itemize}

\emph{Acknowledgements:} The author would like to thank Prof. Dr. Miles Simon for his support and advice. I would also like to thank 
Prof. Dr. Guofang Wang for useful suggestions and comments on an earlier version of this work.

\section{Definitions and auxiliary identities}
In \cite{k} Kosovski\u{\i} introduces a smooth structure on $M$ (Fermi coordinates) relative to which $M_0$ and $M_1$ are smooth submanifolds.
Moreover, the coefficients of 
the metric $g$ on $M$, which is defined by $g|_{M_i}=g_i$, $i=0,1$, are continuous (cf. \cite{k}, Lemma 3.1).
Throughout this work, we will constantly make use of the properties of that structure. To that end, we shall repeat the construction here:

First we cover $\Gamma$ with coordinate charts  $(x^1,\ldots,x^{n-1})$. If the distance $d$ is small enough,
the hypersurfaces $\Gamma(d)$ equidistant to $\Gamma$ are smooth. For a point $p\in M_0$ near $\Gamma$ we put $x^n(p)=d(p,\Gamma)$, and 
$(x^1(p),\ldots,x^{n-1}(p))$ are the same as the coordinates of the point of $\Gamma$ closest to $p$. On $M_1$ we repeat this construction
with $x^n(p)=-d(p,\Gamma)$. In these coordinates the metric tensor $g$ on $M$ defined above is of the form
\begin{eqnarray}\label{matrix}
\begin{pmatrix}
  g_{1,1}   & \cdots & g_{1,n-1}   & 0\\
  \vdots    & \ddots & \vdots      & \vdots \\
  g_{n-1,1} & \cdots & g_{n-1,n-1} & 0  \\
  0         & \cdots &        0    & 1 
 \end{pmatrix}
\end{eqnarray}
The coordinate charts $(x^1,\dots,x^n)$ give us a smooth structure on $M=M_0\cup_\phi M_1$ which we will work with in what follows.
All computations near $\Gamma$ wiil be carried out in these coordinates, unless noted differently.
\begin{notation}\label{normal def}
On $M_0$ we put $\partial_i=\frac{\partial}{\partial x^i}$ for $1\leq i\leq n-1$ and $N=\frac{\partial}{\partial x^n}$, 
where $(x^1,\ldots,x^n)$ is the coordinate chart introduced above. Note that $N(p)$ is a well defined smooth vectorfield near $\Gamma$, which
is normal to the hypersurface $\Gamma(d(p))$ equidistant to $\Gamma$ and containing $p$.
\end{notation}
 
The following lemma uses the above construction to smoothly extend the metric $g_1$ from $M_1$ to a small neighborhood
of $\Gamma$ on $M_0$.

\begin{lemma}[\cite{k}, Lemma 3.1]\label{strichmetrik}
The metric $g_1$ smoothly extends to a small neighborhood of $\Gamma$ in $M_0$
 in such a way that the hypersurfaces equidistant to $\Gamma$ with respect to the extended metric 
$g'_1$ and the metric $g_0$ coincide.
\begin{proof}
In coordinates defined above the metric $g_1$ on $M_1$ is of the same form as in \eqref{matrix}. 
Locally in a small enough coordinate neighborhood $U$ of some point of $\Gamma$ we may smoothly extend $(g_1)_{ij}$, $1\leq i,j\leq n-1$
to $U\cap M_0$ and put $(g'_1)_{in}=\delta_{in}$. We then cover $\Gamma$ by finitely many such neighborhoods and define $g'_1$ near $\Gamma$
using a subordinate partition of unity. One easily checks that the obtained metric has the desired property.
\end{proof}
\end{lemma}
Throughout this work we will use the following
\begin{notation}\label{optens}
Given a $(0,2)$ tensor $A$ on $TM$ we denote by $\mathbf A$ the corresponding linear endomorphism of $TM$ satisfying
\[A(v,w)=\langle v,\mathbf A w\rangle_g\]
If $\{e_1,\dots,e_n\}$ is a basis of $TM$ at some point $p\in M$ and $\mathbf A e_i=\mathbf A_i^j e_j$, then
$\mathbf A_i^j=A_{ki}g^{kl}$, where $A_{ki}=A(e_k,e_i)$ and $(g^{kl})_{1\leq k,l\leq n}$ is the inverse of the matrix $(g(e_k,e_l))_{1\leq k,l,\leq n}$. The operator $\mathbf A$ is self-adjoint iff the tensor $A$ is symmetric.
\end{notation}

\begin{deflem}[The operator $\mathbf L$, cf. \cite{k}, 3.4 and 3.5]\label{defnL}
Let $L$ be the sum of the second fundamental forms on $\Gamma$ with respect to the inward normals on $M_0$ and $M_1$
(or the difference of the second fundamental forms with respect to the common normal $N$), and let $\mathbf L$ be the corresponding selfadjoint operator on $T\Gamma$,
 i.e. $L(\cdot ,\cdot)=\langle\cdot,\mathbf L \cdot\rangle_0$.\\
\\ 
 In a small neighborhood of $\Gamma$
the operator $\mathbf L $ extends to $TM_0$ so that $\mathbf L N=0$ and $\nabla_N\mathbf L =0$.
\begin{proof}
For a point $p\in \Gamma$ we may extend $\mathbf L$ to $T_pM_0$ by linearity such that $\mathbf L N=0$, and for $X\in TM_0$ 
we use parallel transportation $P$ along the integral curves of the vector field $N$ and put
$\mathbf L X:=P^{-1}\mathbf L PX$
\end{proof}
\end{deflem}
\noindent
Note that if the initial operator is positive semidefinite, then so is its extention: 
\[\langle X,\mathbf L X\rangle_0=\langle X,P^{-1}\mathbf L PX\rangle_0=\langle PX,\mathbf L PX\rangle_0\geq 0\]

The following $C^\infty([0,\infty),\mathbb R)$-functions will be used to modify the metric $g_0$ near $\Gamma$:

\begin{defin}[Auxiliary functions $f_\delta$, $F_\delta$ and $\mathcal F_\delta$, cf. \cite{k}, 3.3]\label{fdeltadef}
Let  $f_\delta$ be a function on  $\mathbb R_{\geq 0}$ with the following properties:
\begin{eqnarray*}
f_\delta(x)= 1-\frac{x}{\delta^4}\quad  &\textnormal{ if } & x\in [0,\delta^4] \\
-\delta^2 \leq f_\delta \leq 0
\textnormal{ and } f'_\delta(x)\leq \delta \quad  &\textnormal{ if } & x\in [\delta^4,\delta] \\
f_\delta(x)= 0 \quad  &\textnormal{ if } & x\in [\delta,+\infty) 
\end{eqnarray*}
\noindent
and
\[\int_0^\delta f_\delta(t)dt=0\]
\noindent 
We put
\[F_\delta(x):=\int_0^xf_\delta(t)dt\]
\[\mathcal F_\delta(x):=\int_0^xF_\delta(t)dt\]
\end{defin}
\begin{figure}[ht]
	\centering
  \includegraphics[scale=0.3]{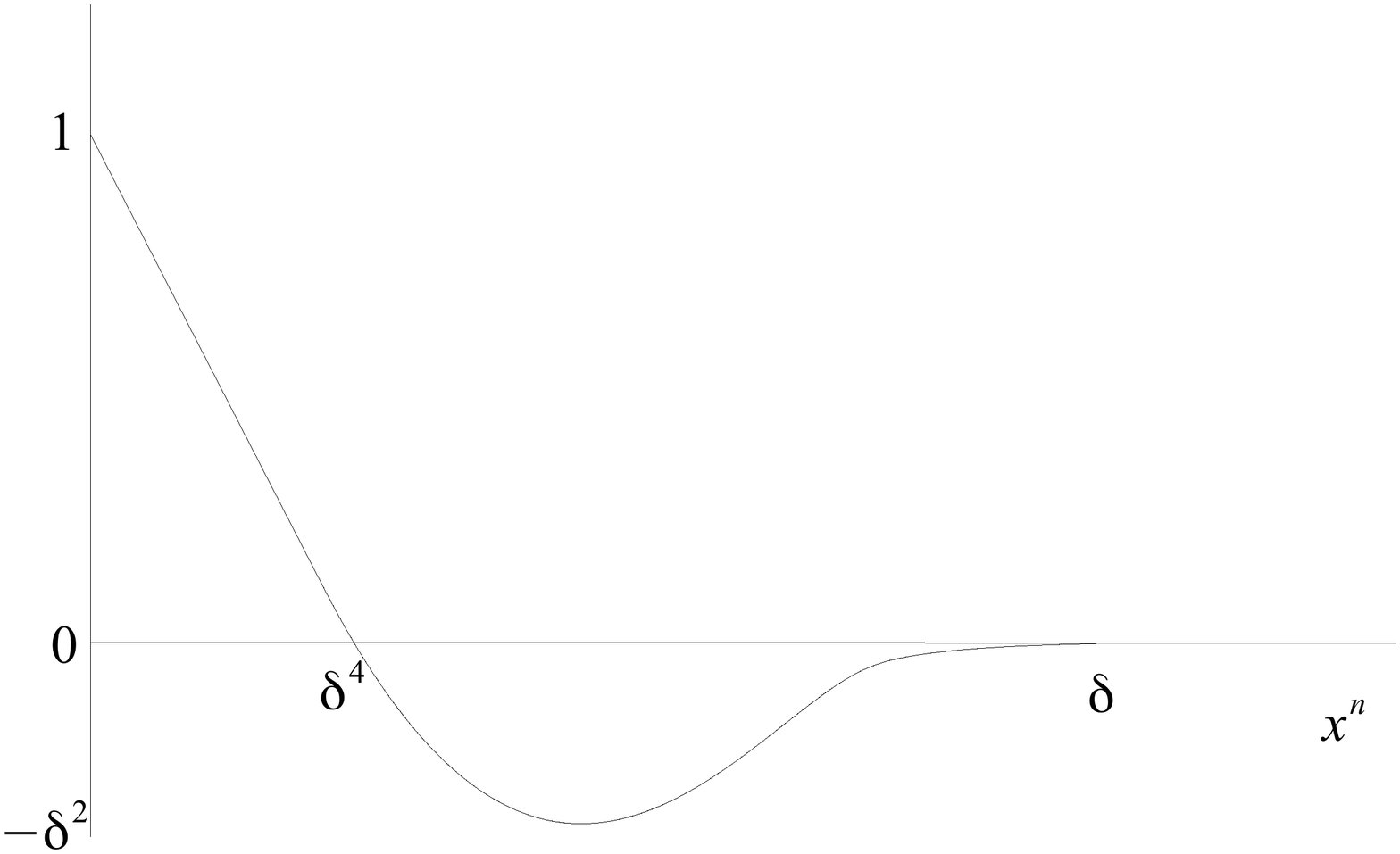}
	\caption{The function $f_\delta$}
	\label{fig1}
\end{figure}

%\noindent
%and
%\[\int_0^\delta f_\delta(t)dt=0\]
%\begin{figure} [ht]
%\centering
%\scalebox{0.25}{\includegraphics{schaubild.eps}}
%\caption{The function $f_\delta$}
%\end{figure}

%\newpage
\begin{notation}[Projection operators]\label{proj op}
 Let 
 \[\mathbf P^T :TM_0\rightarrow T\Gamma(d)\subset TM_0\]
 and
 \[\mathbf P^N:TM_0\rightarrow (T\Gamma(d))^\perp\subset TM_0\]
be the projection operators. 
The coefficients of the corresponding $(0,2)$-tensors (with respect to the coordinates chosen above) are 
\[( P^T)_{ij}=
\begin{pmatrix}
  (g_{ij})_{1\leq i,j\leq n-1}     & 0\\
  0    & 0
 \end{pmatrix}
 \]
 and
 $( P^N)_{ij}=\delta_{in}\delta_{jn}$
 \end{notation}

\begin{defin}[The modified metric $g_\delta$, \cite{k}, 3.6]

Let $\mathbf I$ denote the identity on $TM_0$. We define the self-adjoint endomorphism $\mathbf G _\delta$ by
\begin{eqnarray}\label{defG}
\mathbf G_\delta =\mathbf I +2F_\delta(x^n)\mathbf L -2C\mathcal F_\delta(x^n)\mathbf P^T 
\end{eqnarray}
and the modified inner product $\langle\cdot,\cdot\rangle_\delta$ on $TM_0$  by
\[\langle\cdot,\cdot\rangle_\delta=\langle\cdot,\mathbf G_\delta \cdot\rangle_0\]
i.e. in coordinates we have
\[g^\delta_{ij}=g^0_{ij}+2F_\delta (x^n)L_{ij}-2C\mc F_\delta(x^n)(P^T)_{ij}\]
\end{defin}

\noindent
The constant $C$ the definition of $\mathbf G_\delta$ is to be chosen later. Note that the operator 
$\mathbf G_\delta$ is well defined globally on $M_0$, since by our choice of coordinates we have 
$x^n(p)=d_{g_0}(p,\Gamma)$ for a point $p$ near $\Gamma$, and $\Gamma$ is compact by assumption.

\begin{remark}\label{gprops}
$\mathbf G_\delta$ has the following properties: \\
\\
(i) As $\delta$ tends to zero, $\mathbf G_\delta$ converges to $\mathbf I$ uniformly on $M_0$.
\\
\\
(ii) The coefficients of the metric $g_{(\delta)}$ which is defined as $g_{\delta}$ on $M_0$ and $g_1$ on $M_1$ belong to $W^{2,\infty}_{loc}$
\\
\\ 
(iii) The hypersurfaces $\Gamma(d_{g_0})$ and $\Gamma(d_{g_\delta})$ coincide.
\begin{proof}
$(i)$: $\mathbf L$ and $\mathbf P^T$ are bounded near $\Gamma$, and $F_\delta,\mc F_\delta\rightarrow 0$ uniformly as $\delta\rightarrow 0$.\\
\\
$(ii)$:  On $\Gamma$ we have
\[\partial_kg^\delta_{ij}=\partial_kg^0_{ij}=\partial_k g^1_{ij}\]
for $k=1,\dots, n-1$, since $g_0=g_1$ on $\Gamma$. In a point $p\in \Gamma$, using
$L^0_{ij}=-\la \nabla^0_{\partial_i} N,\partial_j\ra_0$ and $L^1_{ij}=\la \nabla^1_{\partial_i} N,\partial_j\ra_1$
one computes
\[\partial_n g_{ij}^0=-2L^0_{ij}\] and
\[\partial_n g_{ij}^1=2L^1_{ij}\] 
Thus, on $\Gamma$ we have
\[\partial_n g^\delta_{ij}=\partial_n g^0_{ij}+ 2L_{ij}=2(L_{ij}-L^0_{ij})=2L^1_{ij}=\partial_n g^1_{ij}\]
which implies that the first derivatives of $g_{(\delta)}$ are continuous on $M$. Since $\Gamma\subset M$ is a smooth submanifold and 
$g_{(\delta)}$ is smooth on $M_0$ and $M_1$, respectively, we have $g_{(\delta)}\in W^{2,\infty}_{loc}$  
\\
\\
$(iii)$ Note that  $g_\delta(\partial_i,N)=\delta_{in}=g_0(\partial_i,N)$ due to $\mathbf LN=0=\mathbf P^TN$. Therefore 
$N$ is also normal with respect to $g_\delta$ to the hypersurfaces $\Gamma(d_{g_\delta})$, which implies
$d_{g_\delta}(\cdot,\Gamma)=d_g(\cdot,\Gamma)$.
\end{proof}
\end{remark}

\begin{defin}\label{def approx}
For two endomorphisms $\mathbf S_\delta, \mathbf T_\delta$ of $TM_0$ which depend on $\delta$ we say 
that \[\mathbf S_\delta\approx \mathbf T_\delta\]
iff $\mathbf S_\delta|_\Gamma=\mathbf T_\delta|_\Gamma$ and all eigenvalues of $\mathbf S_\delta-\mathbf T_\delta$
uniformly tend to zero on compact subsets of $M_0$ as $\delta\rightarrow 0$. 
\\
For two vectorfields $X_\delta,Y_\delta$ on $M_0$ we say that $X_\delta\approx Y_\delta$
iff $X_\delta|_\Gamma=Y_\delta|_\Gamma$ and $\|X_\delta-Y_\delta\|_0\rightarrow 0$ uniformly on compact subsets as $\delta\rightarrow 0$.
\end{defin}
\noindent
Note that $\mathbf S_\delta\approx \mathbf T_\delta$ ($X_\delta\approx Y_\delta$) holds iff in local coordinates
$ (S_{\delta})_{ij}=(T_{\delta})_{ij}$ on $\Gamma$ and
$|(S_\delta)_{ij}-(T_\delta)_{ij}|\rightarrow 0$ 
($X^{i}_{\delta}=Y^{i}_{\delta}$ on $\Gamma$ and $|X^{i}_{\delta}-Y^{i}_{\delta}|\rightarrow 0$).

%\begin{remark}
%$\mathbf S_\delta\approx \mathbf T_\delta$ ($\mc S_\delta\approx \mc T_\delta$) holds iff all coordinate functions of the 
%corresponding 2-tensors (4-tensors) coincide on $\Gamma$ 
%and their difference tends to zero uniformly on compact subsets.
%\end{remark}

\begin{lemma}[Auxiliary identities, cf. \cite {k}, Lemma 6.1, 6.2, 6.3]\label{approx}
Let \[X,Y\in \{\partial_1,\ldots,\partial_{n-1}\}\subset T\Gamma(d)\subset TM_0 \]
and \[N=\partial_n\in (T\Gamma(d))^\perp\subset TM_0\]
Then the following (approximate and exact) identities hold

\begin{eqnarray}\label{approx61}
\mathbf G_\delta & \approx & \mathbf I,\quad \nabla_X\mathbf G_\delta \approx 0,\quad \nabla_N\mathbf G_\delta \approx 2f_\delta(x^n)\mathbf L\nonumber  \\
\nabla  _X\nabla_N  \mathbf G_\delta & \approx & 2 f_\delta(x^n) \nabla_X \mathbf L  \\
\nabla_N\nabla_N \mathbf G_\delta &\approx & 2f'_\delta(x^n)\mathbf L - 2Cf_\delta(x^n) \mathbf P^T \nonumber
\end{eqnarray}

\begin{eqnarray}\label{approx62}
\langle \nabla^\delta_X N,Y\rangle_\delta = \langle \nabla^\delta_N X,Y\rangle_\delta =
 \frac 12 \bigl(
          \langle \nabla_N X,\mathbf G_\delta Y\rangle
         +\langle X,\mathbf G_\delta \nabla_N Y\rangle
         +\langle X,(\nabla_N \mathbf G_\delta) Y\rangle           
          \bigr) 
\end{eqnarray}

\begin{eqnarray}\label{approx64}
\nabla^\delta_NN=0
\end{eqnarray}
\begin{eqnarray}\label{approx65}
\nabla^\delta_NX=\nabla^\delta_X N\approx \nabla_X N +f_\delta(x^n)\mathbf L X
\end{eqnarray}

\begin{eqnarray}\label{grad approx}
\mathbf P^T(\nabla^\delta_XY)\approx\mathbf P^T(\nabla_XY) 
\end{eqnarray}
\begin{proof}
Detailed proofs of these  identities are given in \cite{k}.
For the convenience of the reader we prove one of the identities in \eqref{approx61}. $\nabla_X \mathbf G_\delta \approx 0$:

By definition of $\mathbf G_\delta$ we have
\begin{eqnarray}\label{derg}
\nabla_X \mathbf G_\delta =\nabla_X \mathbf I + 2\nabla_X(F_\delta(x^n)\mathbf L) -2C\nabla_X(\mathcal F_\delta(x^n)\mathbf P^T) 
\end{eqnarray}
Let $\xi\in TM_0$. We compute
\begin{eqnarray*}
\nabla_X \mathbf I (\xi)=\nabla_X (\mathbf I\xi)-\mathbf I(\nabla_X \xi)=0
\end{eqnarray*}
and
\begin{eqnarray*}
\nabla_X(F_\delta(x^n)\mathbf L)\xi &=& \nabla_X(F_\delta(x^n)\mathbf L\xi)-F_\delta(x^n)\mathbf L\nabla_X\xi\\
&=& X(F_\delta(x^n))\mathbf L \xi+ F_\delta(x^n)\nabla_X(\mathbf L \xi) - F_\delta(x^n)\mathbf L\nabla_X\xi
\end{eqnarray*}

The first term of the last expression vanishes since $X\in\{\partial_1,\ldots,\partial_{n-1}\}$ and $F_\delta$ only depends on $x^n$. The next two terms tend to zero as $\delta\rightarrow 0$
by definition of $F_\delta$.
By a similar computation one verifies that the last term in \eqref{derg} is 
$\approx 0$ as well.\\

\end{proof}
\end{lemma}

\section{The Riemannian curvature operator of $g_\delta$}

In this section we compare the Riemannian curvature operators of $g_\delta$ and $g_0$ on $M_0$ (cf. \S\S~7-8 of \cite{k}). \\

Let us first recall that given a finite dimensional vectorspace $V$ one has the following 
connection between $(0,4)$-tensors on $V$ and linear operators and bilinear forms on $\Lambda^2V$:
Any $(0,4)$-tensor $\{T_{ijkl}\}$ which is antisymmetric in $i,j$ and $k,l$, respectively, induces a bilinear form $\mc T$ on $\Lambda^2V$ via
\[\mc T(e_i\we e_j,e_k\we e_l):=T(e_i,e_j,e_k,e_l)=T_{ijkl}\]
where $e_1,\dots,e_n$ is some basis of $V$. The antisymmetries of $T$ ensure that 
\[\mc T(e_i\we e_j,e_k\we e_l)=-\mc T(e_j\we e_i,e_k\we e_l)=-\mc T(e_i\we e_j,e_l\we e_k)\]
If in addition $T_{ijkl}=T_{klij}$ then the induced bilinearform $\mc T$ is symmetric. For arbitrary $\alpha,\beta\in \Lambda^2V$,
$\alpha=\sum_{i<j}\alpha^{ij}e_i\we e_j=\alpha^{ij}e_i\otimes e_j$, $\beta=\sum_{i<j}\beta^{ij}e_i\we e_j=\beta^{ij}e_i\otimes e_j$
($\alpha^{ij}=-\alpha^{ji}$ and $\beta^{ij}=-\beta^{ji}$)
one computes
\[\mc T(\alpha,\beta)=\frac 14 T_{ijkl}\alpha^{ij}\beta^{kl}\]
(here and in hat follows we make use of the summation convention). As mentioned in the introduction, an inner product $g$ on $V$ induces
an inner product $\mc I^g$ on $\Lambda^2V$:
\[\mc I^g(e_i\we e_j,e_k\we e_l)=g_{ik}g_{jl}-g_{jk}g_{il}\]
Using this inner product we may identify linear operators and bilinear forms on $\Lambda^2V$ by putting
\[\mc I^g(e_i\we e_j,\mc T(e_k\we e_l))=\mc T(e_i\we e_j,e_k\we e_l)\]
The bilinear form is symmetric iff the operator is self-adjoint.

Conversely, any bilinear form $\mc T$ on $\Lambda^2V$ (or the corresponding linear operator) induces a $(0,4)$-tensor on $V$ via
\[T(e_i,e_j,e_k,e_l):=\mc T(e_i\we e_j,e_k\we e_l)=\mc I^g(e_i\we e_j,\mc T(e_k\we e_l))\]
The such defined tensor has the symmetries $T_{ijkl}=-T_{jikl}=-T_{ijlk}$, and if in addition the bilinear form is symmetric, then we also have
$T_{ijkl}=T_{klij}$.

In view of these identifications, in what follows we will often switch between operators and bilinear forms on $\Lambda^2(TM)$ and $(0,4)$-tensors on $TM$.
\\
\\
In Section \ref{ricci section}
% and \ref{scalar section} 
we will use the following 
\begin{lemma}\label{traces}
Let $\mc T$ be a bilinear form on $\Lambda^2 V $ and $(T_{ijkl})$ the corresponding $(0,4)$-tensor on $V$. If $\mc T$ is positive semidefinte,
then so is the bilinear form $\tr_{24} T:=g^{jl}T(\cdot,e_j,\cdot, e_l): V\times V\rightarrow \mathbb R$.
\begin{proof}
Let $\{e_1,\dots,e_n\}$ be a basis of $V$ such that $g_{ij}=g(e_i,e_j)=\delta_{ij}$. Let $\xi=\xi^k e_k\in V$. For every $1\leq j\leq n$ we define the 2-vector 
$\alpha_j=\xi^k e_k\we e_j\in \Lambda^2 V$. By assumption $\mc T(\alpha_j,\alpha_j)\geq 0$. We  compute 
\begin{eqnarray*}
\tr_{24}T(\xi,\xi)&=& g^{jl}T(\xi,e_j,\xi,e_l)=\sum_{j}\xi^k\xi^lT(e_k,e_j,e_l,e_j)\\
&=&\sum_j \xi^k\xi^l\mc T(e_k\we e_j,e_l\we e_j) =\sum_j \mc T(\alpha_j,\alpha_j)\geq 0
\end{eqnarray*}
\end{proof}
\end{lemma}
\noindent
Note that this lemma also holds if we replace  $\tr_{24}T$ by $\tr_{13}T$.
We will also make use of the Kulkarni-Nomizu product on $\textnormal{End}(TM)$, which is defined as follows:\\
\\
The Kulkarni-Nomizu product of two linear endomorphisms $\mathbf A$, $\mathbf B$ of  $V$ 
is the linear endomorphism $\mathbf A\we\mathbf B: \Lambda^2 V\rightarrow \Lambda^2V$, which is defined as
\begin{eqnarray*}
(\mathbf A\we \mathbf B)(e_i\we e_j):=\frac 12\bigl(\mathbf A(e_i)\we\mathbf B(e_j)+\mathbf B(e_i)\we \mathbf A(e_j)\bigr)
\end{eqnarray*}
for basis vectors $e_i\we e_j$, and extends to $\Lambda^2 V$ by linearity.
The factor $\frac 12$ ensures that we have $\id_V\we\id_V=\id_{\Lambda^2V}$. 
The corresponding bilinear form on $\Lambda^2V$ is given by
\begin{eqnarray*}
\mathbf A\we\mathbf B(e_i\we e_j,e_k\we e_l)&:=&\mc I^g\bigl(e_i\we e_j,(\mathbf A\we \mathbf B)(e_k\we e_l)\bigr)\\
&=&\frac 12(A_{ik} B_{jl}-A_{jk} B_{il}+B_{ik} A_{jl}-B_{jk}  A_{il})
\end{eqnarray*}
where  $A,B$ are the bilinear forms on $V$ corresponding with $\mathbf A,\mathbf B$ (cf. Notation \ref{optens}).
Note that the induced $(0,4)$-tensor $\{(\mathbf A\we \mathbf B)_{ijkl}\}$ is antisymmetric in $i,j$ and $k,l$, respectively. If in addition
$A$ and $B$ are symmetric, then we also have the symmetry $(\mathbf A\we \mathbf B)_{ijkl}=(\mathbf A\we \mathbf B)_{klij}$.
\\
\\
Throughout this work we will frequently make use of the following
\begin{lemma}\label{produkt geq 0}
Let $\mathbf A,\mathbf B$ be two self-adjoint endomorphisms of $(V,g)$. If $\mathbf A,\mathbf B\geq 0$ (in the sense of eigenvalues) then $\mathbf A\we\mathbf B\geq 0$.
\begin{proof}
It suffices to show that $(\mathbf A\we\mathbf B)(\alpha,\alpha)\geq 0$ for any $\alpha\in \Lambda^2V$.
Let $\{e_1,\dots,e_n\}$ be an orthonormal basis of $(V,g)$ such that $A_{ij}=\lambda_i\delta_{ij}$ with respect to this basis (where $\lambda_i\geq 0$ by assumption),
and $\alpha=\sum_{i<j}\alpha^{ij}e_i\we e_j=\alpha^{ij}e_i\otimes e_j\in \Lambda^2V$.
% (here and in what follows we use the summation convention)
%Note that any bilinear form $\mc T$ on $\Lambda^2V$ satisfies 
As mentioned above, a bilinear form $\mc T$ induced by a $(0,4)$-tensor $T$ satisfies
\[\mc T(\alpha,\beta)=\frac 14T_{ijkl}\alpha^{ij}\beta^{kl}\]
where $T_{ijkl}:=\mc T(e_i\we e_j,e_k\we e_l)$. 
We compute
\begin{eqnarray*}
(\mathbf A\we \mathbf B)(\alpha,\alpha)&=&\frac 18(A_{ik} B_{jl}\alpha^{ij}\alpha^{kl}-A_{jk} B_{il}\alpha^{ij}\alpha^{kl}+B_{ik} A_{jl}\alpha^{ij}\alpha^{kl}-B_{jk}  A_{il}\alpha^{ij}\alpha^{kl})\\
&=&\frac 12A_{ik}B_{jl}\alpha^{ij}\alpha^{kl}\\
&=&\frac 12\lambda_i B_{jl}\alpha^{ij}\alpha^{il}\geq 0
\end{eqnarray*}
where we used $\alpha^{ij}=-\alpha^{ji}$ and the fact that for every fixed $i$ we have $B_{jl}\alpha^{ij}\alpha^{il}\geq 0$ by assumption.
\end{proof}
\end{lemma}

Let us now consider the Riemannian curvature operator of $g_\delta$. 
For ease of notation here and in what follows we shall omit the index 0 for quantities related to $M_0$. For example, we write
$\langle\cdot,\cdot\rangle$ for $\langle\cdot,\cdot\rangle_0$ and $\mc R$ for $\mathcal R_0$.
We define $\mc S_\delta\approx\mc T_{\delta}$ for selfadjoint operators on $\Lambda^2(TM_0)$ in a similar way as in Definition \ref{def approx}.
The main result of this section is 
\begin{thm}\label{operatorzerlegung}
Let $\mathcal R_\delta=\mathcal R(g_\delta)$. Then
\begin{eqnarray}\label{RdeltaOp}
\mathcal R_\delta &\approx& \mathcal R-f_\delta^2\mc A + f_\delta \mc B -2f_\delta'\mc L +2f_\delta^2\mc L^2 + 2Cf_\delta\hat {\mc I}
\end{eqnarray}
holds on $M_0$, where 
\begin{eqnarray*}
\mc A &:=&\mathbf L\we\mathbf L \\
\mc L &:=&\mathbf L\we \mathbf P^N \\
\mc L^2 &:=& \mathbf L^2\we \mathbf P^N\\
\hat{\mc I} &:=& \mathbf P^T\we\mathbf P^N
\end{eqnarray*}
(cf. Notation \ref{proj op} for the definitions of $\mathbf P^T$ and $\mathbf P^N$),
and $\mc B$ is a smooth operator on $\Lambda^2(TM)$ depending on $\mathbf L$ which we will define later.

\end{thm}
\begin{lemma}\label{kein n}
For $i,j,k,l\in\{1,\dots,n-1\}$ we have
\begin{eqnarray}\label{kein n glchg}
R^\delta_{ijkl} \approx 
R_{ijkl} -f^2_\delta (\mathbf L\we\mathbf L)_{ijkl} -2 f_\delta (\mathbf L\we \nabla N)_{ijkl}
\end{eqnarray}
where $\nabla N$ is the endomorphism $X\in TM \mapsto \nabla_XN\in TM$ (recall that $N$ is the unit vector field orthogonal to the hypersurfaces of $M_0$ equidistant to $\Gamma$, 
cf. Notation \ref{normal def}).
\begin{proof}
We proceed as in \cite{k}, Lemma 7.1.
Let $p\in M_0$ be a point near $\Gamma$ and $d=\dist(x,\Gamma)=x^n(p)$. Let $k,l\leq n-1$. Recall that by Definition \ref{defG} we have
\[g^\delta_{kl}=g_{kl}+2F_\delta(x^n)L_{kl}-2C\mc F_\delta(x^n)g_{kl}\] 
Therefore, for $i,j\leq n-1$
\[\partial_i g^\delta_{kl}\approx \partial_i g_{kl}\] and
\[ \partial_i\partial_j g^\delta_{kl}\approx \partial_i\partial_j g_{kl}\]
and thus 
\[R^\delta_{\Gamma(d)}
\approx 
R_{\Gamma(d)}\]
Using the Gauss theorem and \eqref{approx65} we compute in $p$
\begin{eqnarray*}
R^\delta_{ijkl}&=&\langle R^\delta(\partial_i,\partial_j)\partial_k,\partial_l\rangle_\delta\\
&=&\langle R_{\Gamma(d)}^\delta(\partial_i,\partial_j)\partial_k,\partial_l\rangle_\delta
-\langle\nabla^\delta_{\partial_i}N,\partial_k\rangle_\delta  \langle\nabla^\delta_{\partial_j}N,\partial_l\rangle_\delta
+ \langle\nabla^\delta_{\partial_j}N,\partial_k\rangle_\delta \langle\nabla^\delta_{\partial_i}N,\partial_l\rangle_\delta 
\\
&\approx&
 \langle R_{\Gamma(d)}(\partial_i,\partial_j)\partial_k,\partial_l\rangle
-(\langle\nabla_{\partial_i}N,\partial_k\rangle +f_\delta\langle\partial_i,\mathbf L\partial_k\rangle)
(\langle\nabla_{\partial_j}N,\partial_l\rangle +f_\delta\langle\partial_j,\mathbf L\partial_l\rangle)\\
&+&(\langle\nabla_{\partial_j}N,\partial_k\rangle +f_\delta\langle\partial_j,\mathbf L\partial_k\rangle) 
   (\langle\nabla_{\partial_i}N,\partial_l\rangle +f_\delta\langle\partial_i,\mathbf L\partial_l\rangle)\\
&=&\langle R_{\Gamma(d)}(\partial_i,\partial_j)\partial_k,\partial_l\rangle
-\langle\nabla_{\partial_i}N,\partial_k\rangle  \langle\nabla_{\partial_j}N,\partial_l\rangle
+\langle\nabla_{\partial_j}N,\partial_k\rangle  \langle\nabla_{\partial_i}N,\partial_l\rangle\\
&-&f^2_\delta(\langle\partial_i,\mathbf L\partial_k\rangle \langle\partial_j,\mathbf L\partial_l\rangle
- \langle\partial_j,\mathbf L\partial_k\rangle \langle\partial_i,\mathbf L\partial_l\rangle) \\
&-& f_\delta(\langle\partial_i,\mathbf L\partial_k\rangle   \langle\nabla_{\partial_j}N,\partial_l\rangle
            - \langle\partial_j,\mathbf L\partial_k\rangle   \langle\nabla_{\partial_i}N,\partial_l\rangle 
            + \langle\nabla_{\partial_i}N,\partial_k\rangle   \langle\partial_j,\mathbf L\partial_l\rangle
            -\langle\nabla_{\partial_j}N,\partial_k\rangle   \langle\partial_i,\mathbf L\partial_l\rangle)
 \end{eqnarray*}
\end{proof}
\end{lemma}

\begin{lemma}\label{ein n}
For $i,j,l\in\{1,\dots,n-1\}$ we have
\begin{eqnarray}
R^\delta_{ijnl}
\approx R_{ijnl}+f_\delta \bigl(\langle \partial_i,(\nabla_{\partial_j} \mathbf L )\partial_l\rangle
                               -\langle \partial_j,(\nabla_{\partial_i} \mathbf L )\partial_l\rangle\bigr)
\end{eqnarray}

\begin{proof}
We proceed as in \cite{k}, Lemma 7.3.
Let $i,j,l\in \{1,\ldots,n-1\}$. By definition of the Riemannian curvature tensor we have
\begin{eqnarray}
\langle R^\delta(\partial_i,\partial_j)\partial_n,\partial_l\rangle_\delta &=&
\langle \nabla^\delta_{\partial_j} \nabla^\delta_{\partial_i} N,\partial_l \rangle_\delta
-\langle \nabla^\delta_{\partial_i} \nabla^\delta_{\partial_j} N,\partial_l \rangle_\delta
 \label{someref}\\
&=&\partial_j \langle\nabla^\delta_{\partial_i} N,\partial_l\rangle_\delta   -  \partial_i \langle\nabla^\delta_{\partial_j} N,\partial_l\rangle_\delta 
-\langle \nabla^\delta_{\partial_i} N,\nabla^\delta_{\partial_j}\partial_l  \rangle_\delta\nonumber
+\langle \nabla^\delta_{\partial_j} N,\nabla^\delta_{\partial_i}\partial_l  \rangle_\delta
\end{eqnarray}
\noindent
1) For the first two terms on the right hand side we compute
using \eqref{approx62}:

\begin{eqnarray*}
&{}&\partial_j \langle\nabla^\delta_{\partial_i} N,\partial_l\rangle_\delta  -  \partial_i \langle\nabla^\delta_{\partial_j} N,\partial_l\rangle_\delta\\
&=&\frac 12 \partial_j(\langle\nabla_N \partial_i,\mathbf G_\delta \partial_l\rangle
                        +\langle\partial_i,\mathbf G _\delta \nabla_N \partial_l\rangle
                        +\langle \partial_i,(\nabla_N \mathbf G_\delta)\partial_l  \rangle ) \\ 
&-& \frac 12 \partial_i(\langle\nabla_N \partial_j,\mathbf G_\delta \partial_l\rangle
                        +\langle\partial_j,\mathbf G _\delta \nabla_N \partial_l\rangle
                        +\langle \partial_j,(\nabla_N \mathbf G_\delta)\partial_l  \rangle )
\end{eqnarray*}

After termwise differentiation we get three different types of terms:\\
\\
a) Terms in which $\mathbf G _\delta$ is not differentiated: Since $\mathbf G_\delta \approx \mathbf I$, 
their sum is approximately equal to
\begin{eqnarray*}
%\label{ggg}
\partial_j \langle\nabla_{\partial_i} N,\partial_l\rangle  -  \partial_i \langle\nabla_{\partial_j} N,\partial_l\rangle
\end{eqnarray*} 
\noindent
b) Terms in which $\mathbf G_\delta$ is differentiated only with respect to $\partial_i$ are approximately
equal to $0$ by \eqref{approx61} since we have $i\in\{1,\dots,n-1\}$. 
\\
\\
c) Terms which involve mixed derivatives  of $\mathbf G_\delta $ with respect to both $\partial_i$ and $N$.
In view of \eqref{approx61} their sum is
\begin{eqnarray*}
%\label{g}
&\approx& f_\delta\bigl(
\langle \nabla_{\partial_j}\partial_i,\mathbf L \partial_l\rangle
               +\langle \partial_i,(\nabla_{\partial_j} \mathbf L )\partial_l\rangle
               + \langle\partial_i,\mathbf L (\nabla_{\partial_j}\partial_l )\rangle\\
&-&\langle \nabla_{\partial_i}\partial_j,\mathbf L \partial_l\rangle 
                -\langle \partial_j,(\nabla_{\partial_i} \mathbf L )\partial_l\rangle
                -\langle\partial_j,\mathbf L (\nabla_{\partial_i}\partial_l)\rangle \bigr)\nonumber\\
&=& f_\delta\bigl(\langle \partial_i,(\nabla_{\partial_j} \mathbf L )\partial_l\rangle
               + \langle\partial_i,\mathbf L (\nabla_{\partial_j}\partial_l )\rangle
                -\langle \partial_j,(\nabla_{\partial_i} \mathbf L )\partial_l\rangle
                -\langle\partial_j,\mathbf L (\nabla_{\partial_i}\partial_l)\rangle \bigr)
                \nonumber
\end{eqnarray*}
where we used that $\partial_i$ and $\partial_j$ commute. Combining a), b) and c) gives us
\begin{eqnarray}\label{abc}
\partial_j \langle\nabla^\delta_{\partial_i} N,\partial_l\rangle_\delta  -  \partial_i \langle\nabla^\delta_{\partial_j} N,\partial_l\rangle_\delta&\approx&
\partial_j \langle\nabla_{\partial_i} N,\partial_l\rangle  -  \partial_i \langle\nabla_{\partial_j} N,\partial_l\rangle\\
&+&f_\delta\bigl(\langle \partial_i,(\nabla_{\partial_j} \mathbf L )\partial_l\rangle
               + \langle\partial_i,\mathbf L (\nabla_{\partial_j}\partial_l )\rangle
                -\langle \partial_j,(\nabla_{\partial_i} \mathbf L )\partial_l\rangle
                -\langle\partial_j,\mathbf L (\nabla_{\partial_i}\partial_l)\rangle \bigr)\nonumber
 \end{eqnarray}
2) Let us now consider the last two terms on the right hand side of \eqref{someref}.
Since 
$\langle \nabla^\delta_{\partial_i}N,N\rangle_\delta=0$ we have
\[\langle \nabla^\delta_{\partial_i}N,\nabla^\delta_{\partial_j}\partial_l\rangle_\delta
=\langle \nabla^\delta_{\partial_i}N,\mathbf P^T(\nabla^\delta_{\partial_j}\partial_l)\rangle_\delta\]
Therefore, in view of \eqref{approx65} and \eqref{grad approx}
\begin{eqnarray}\label{gg}
&{}&-\langle \nabla^\delta_{\partial_i} N,\nabla^\delta_{\partial_j}\partial_l  \rangle_\delta+\langle \nabla^\delta_{\partial_j} N,\nabla^\delta_{\partial_i}\partial_l  \rangle_\delta
 \nonumber\\
&=&
-\langle \nabla^\delta_{\partial_i} N,\mathbf P^T(\nabla^\delta_{\partial_j}\partial_l ) \rangle_\delta
+\langle \nabla^\delta_{\partial_j} N,\mathbf P^T(\nabla^\delta_{\partial_i}\partial_l  )\rangle_\delta
\\
&\approx&-\bigl(\langle \nabla_{\partial_i} N,\mathbf P^T(\nabla_{\partial_j}\partial_l ) \rangle
+ f_\delta \langle \mathbf L \partial_i, \mathbf P^T(\nabla_{\partial_j}\partial_l)\rangle \bigr)
+\bigl(\langle \nabla_{\partial_j} N,\mathbf P^T(\nabla_{\partial_i}\partial_l ) \rangle
+ f_\delta \langle \mathbf L \partial_j, \mathbf P^T(\nabla_{\partial_i}\partial_l)\rangle \bigr)
\nonumber\\
&\approx &
-\langle \nabla_{\partial_i} N,\nabla_{\partial_j}\partial_l  \rangle
+\langle \nabla_{\partial_j} N,\nabla_{\partial_i}\partial_l  \rangle
+ f_\delta\bigl(\langle\mathbf L \partial_j,\nabla_{\partial_i}\partial_l \rangle
-\langle\mathbf L \partial_i,\nabla_{\partial_j}\partial_l \rangle\bigr)\nonumber
\end{eqnarray}
where in the last line we used $\langle \mathbf L\partial_i,N\rangle=\langle\partial_i,\mathbf L N\rangle=0$.
\\

Combining \eqref{abc} and \eqref{gg} we obtain the desired result.
\end{proof}

\end{lemma}

\begin{lemma}
\label{zwei n}
For $j,l\in\{1,\dots,n-1\}$ we have
\begin{eqnarray*}
R^\delta_{njnl}\approx&& R_{njnl}
-2f_\delta'(\mb L\we\mb P^N)_{njnl}
                                          +2f^2_\delta (\mb L^2\we\mb P^N)_{njnl}
                                          + 2C f_\delta(\mb P^T\we\mb P^N)_{njnl}\\
                                          &&-f_\delta\bigl( \langle\mathbf L \partial_j,\nabla_{\partial_l}N\rangle 
                                          + \langle\nabla_{\partial_j}N,\mathbf L \partial_l\rangle \bigr)
\end{eqnarray*}

\begin{proof}
We proceed as in \cite{k}, Lemma 7.2. Using Lemma \ref{approx} we compute
\begin{eqnarray*}
R^\delta_{njnl}&=&\langle R^\delta(N,\partial_j)N,\partial_l\rangle_\delta\\
&=&\langle\nabla_{\partial_j}^\delta \underbrace{\nabla_N^\delta N}_{=0},\partial_l \rangle_\delta
  -\langle\nabla_N^\delta \nabla_{\partial_j}^\delta N,\partial_l \rangle_\delta\\
&=&-N\langle \nabla_{\partial_j}^\delta N,\partial_l \rangle_\delta + \langle \nabla_{\partial_j}^\delta N,\nabla^\delta_N\partial_l \rangle_\delta\\
&\stackrel{\eqref{approx62}, \eqref{approx65}}{\approx} &
-N\Bigl[
\frac 12\bigl(\langle\nabla_N\partial_j,\mathbf G_\delta\partial_l\rangle + \langle\partial_j,\mathbf G_\delta(\nabla_N\partial_l)\rangle
+\langle\partial_j,(\nabla_N \mathbf G_\delta)\partial_l\rangle\bigr)
\Bigr]\\
&&+\langle\nabla_N\partial_j+f_\delta\mathbf L\partial_j ,\nabla_N\partial_l+f_\delta\mathbf L\partial_l \rangle\\
&\stackrel{\eqref{approx61}}{\approx}&\langle R(N,\partial_j)N,\partial_l\rangle-f_\delta'\langle\partial_j,\mathbf L \partial_l\rangle
                                          +f^2_\delta\langle\mathbf L \partial_j,\mathbf L \partial_l\rangle 
                                          + C f_\delta\langle\partial_j,\partial_l\rangle\\
                                        &&-f_\delta\bigl( \langle\mathbf L \partial_j,\nabla_{\partial_l}N\rangle 
                                          + \langle\nabla_{\partial_j}N,\mathbf L \partial_l\rangle \bigr)\\
\end{eqnarray*}

\end{proof}

\end{lemma}
\noindent
\textbf{Proof of Theorem \ref{operatorzerlegung}}.
Let us define the $(0,4)$-tensor $B$ by
\begin{eqnarray*}
B_{ijkl}&=&-2(\mathbf L\we \nabla N)_{ijkl}\\
             &&+\la\partial_i, (\nabla_{\partial_j}\mathbf L)\partial_l\ra \la \partial_k,N\ra  
             -\la\partial_j, (\nabla_{\partial_i}\mathbf L)\partial_l\ra \la \partial_k,N\ra\\
            && - \la \partial_i, (\nabla_{\partial_j}\mathbf L)\partial_k\ra \la \partial_l,N\ra  
             +\la\partial_j, (\nabla_{\partial_i}\mathbf L)\partial_k\ra \la \partial_l,N\ra\\
           &&+ \la\partial_k, (\nabla_{\partial_l}\mathbf L)\partial_j\ra \la \partial_i,N\ra  
             -\la \partial_l, (\nabla_{\partial_k}\mathbf L)\partial_j\ra \la \partial_i,N\ra\\
           &&- \la\partial_k, (\nabla_{\partial_l}\mathbf L)\partial_i\ra \la \partial_j,N\ra  
             +\la \partial_l, (\nabla_{\partial_k}\mathbf L)\partial_i\ra \la \partial_j,N\ra\\
\end{eqnarray*}
Clearly the tensor  $B$ satisfies $B_{ijkl}=-B_{jikl}=-B_{ijlk}$ and $B_{ijkl}=B_{klij}$, thus inducing 
 a symmetric bilinearform $\mc B$ on $\Lambda^2(TM)$, $\mc B(e_i\we e_j,e_k\we e_l)=B_{ijkl}$ (cf. the discussion in the beginning of this section). 
 In view of Lemmas \ref{kein n} - \ref{zwei n} the the claim 
 \begin{eqnarray}\label{Rdeltaproof}
\mathcal R_\delta &\approx& \mathcal R-f_\delta^2\mb L\we\mb L + f_\delta \mc B -2f_\delta'\mb L\we\mb P^N +2f_\delta^2\mb L^2\we\mb P^N + 2Cf_\delta\mb P^T\we\mb P^N
\end{eqnarray}
follows: Note that since the operators on the right hand side (i.e. their corresponding $(0,4)$-tensors) have the same symmetries as 
the curvature operator, it suffices to consider the following three cases:
\\
\\
Case 1)\\
Let $i,j,k,l\leq n-1$. Obviously in this case  \[(\mb L\we\mb P^N)_{ijkl}=(\mb L^2\we\mb P^N)_{ijkl}= ( \mb P^T\we\mb P^N)_{ijkl}=0\] and 
$B_{ijkl}=-2(\mathbf L\we \nabla N)_{ijkl}$. Thus \eqref{Rdeltaproof} follows by Lemma \ref{kein n}.
\\
\\
Case 2)\\
Let $i,j,l\leq n-1$ and $k=n$. Recall that $L_{in}=0$ for all $i$ and $(P^N)_{in}=0$ for $i\leq n-1$. Therefore we have 
\[ (\mb L\we\mb L)_{ijnl}=(\mb L\we\mb P^N)_{ijnl}=(\mb L^2\we\mb P^N)_{ijnl}= ( \mb P^T\we\mb P^N)_{ijnl}=0\]
Moreover, $(\nabla N)_{in}=\langle\partial_i,\nabla_NN\rangle=0=\langle N,\nabla_{\partial_i}N\rangle =(\nabla N)_{ni}$
and therefore
 \[B_{ijnl}=\la\partial_i, (\nabla_{\partial_j}\mathbf L)\partial_l\ra  
                         -\la\partial_j, (\nabla_{\partial_i}\mathbf L)\partial_l\ra\]
and  \eqref{Rdeltaproof} follows by Lemma \ref{ein n}.     
\\
\\
Case 3)\\
Let $j,l\leq n-1$ and $i=k=n$. Cearly $(\mb L\we\mb L)_{njnl}=0$. As in case 2) we have  $(\mathbf L\we \nabla N)_{njnl}=0$ and thus
\begin{eqnarray*}
B_{njnl}&=&\la N, (\nabla_{\partial_j}\mathbf L)\partial_l\ra  
             -\la\partial_j, (\nabla_N\mathbf L)\partial_l\ra 
           + \la N, (\nabla_{\partial_l}\mathbf L)\partial_j\ra  
             -\la \partial_l, (\nabla_N\mathbf L)\partial_j\ra      \\
             &=&    \la N, (\nabla_{\partial_j}\mathbf L)\partial_l\ra  
           + \la N, (\nabla_{\partial_l}\mathbf L)\partial_j\ra  
\end{eqnarray*}            
where we used that $\nabla_N\mb L=0$ (cf. Lemma \ref{defnL}). Using the fact that $\mb L$ is self-adjoint and $\mb L N=0$ we compute
\begin{eqnarray*}
\la N, (\nabla_{\partial_j}\mathbf L)\partial_l\ra& =&
\la N, \nabla_{\partial_j}(\mathbf L\partial_l)\ra - \la N, \mb L(\nabla_{\partial_j}\partial_l)\ra\\
&=&\partial_j\la N,\mb L\partial_l\ra-\la\nabla_{\partial_j}N,\mb L\partial_l\ra  \\
&=&-\la\nabla_{\partial_j}N,\mb L\partial_l\ra  
\end{eqnarray*}
which gives us
\[B_{njnl}=-(\la\nabla_{\partial_j}N,\mb L\partial_l\ra) +\la\nabla_{\partial_l}N,\mb L\partial_j\ra) \]
\eqref{Rdeltaproof} follows by Lemma \ref{zwei n} and we are done.

\section{The Riemannian curvature operator of $g^{1,}{}'$}
Recall that $g^{1,}{}'$ is the extention of $g^1$ on a small neighborhood of $\Gamma$ in $M_0$, as introduced in Lemma \ref{strichmetrik}.
In this section we compare the Riemannian curvature 
operators on $\Gamma$ with respect to the metrics $g$ and $g^{1,}{}'$ (cf. \cite{k}, \S~9).
\\

We define the selfadjoint operator $\mathbf G_1$ on $TM_0$ by $\langle\cdot,\mathbf G_1\cdot\rangle=\langle\cdot,\cdot\rangle'_1$.

\begin{prop}\label{Rstrich}
Let $\mathcal R_1'$ be the Riemannian curvature operator with respect to $g^{1,}{}'$.
On $\Gamma$ we have
\begin{eqnarray*}
\mc R_1' = \mc R - \mc A + \mc B
+ 2\mc L^2 - \nabla_N^2\mc G_1                                           
\end{eqnarray*}
where $\mc A$, $\mc B$ and $\mc L^2$ are as in Theorem \ref{operatorzerlegung} and $\nabla^2_N\mc G_1:=(\nabla^2_N\mathbf G_1)\we\mathbf P^N$.
\end{prop}
\noindent
In particular, since $\mc R_1'=\mc R_1$ holds on $\Gamma$ independently of the extension $g_1'$, and $\mc R_1\geq \kappa$ by assumption, we have
\begin{eqnarray}\label{onGamma}
\mc R - \mc A + \mc B+ 2\mc L^2 - \nabla_N^2\mc G_1\geq\kappa\mc I
\end{eqnarray}
on $\Gamma$, which is an estimate we will use in the next section.

\begin{proof}
As in \cite{k}, Lemma 9.1, we have to check
the approximate identities for $\mathbf G_1$ which correspond with the ones in \eqref{approx61}. 
For the convenience of the reader we repeat the computations from \cite{k}.
Let $X,Y\in\{\partial_1,\dots,\partial_{n-1}\}$.
On $\Gamma$ we have
\[\mathbf G_1= \mathbf I\]
which implies
\[\quad \nabla_X\mathbf G_1=0\]
on $\Gamma$. Moreover, we have $\nabla_N\mathbf G_1= 2\mathbf L$ on $\Gamma$. Indeed
\begin{eqnarray*}
\langle X,(\nabla_N \mathbf G_1)Y\rangle &=& \langle X,\nabla_N( \mathbf G_1Y)\rangle - \langle X, \mathbf G_1(\nabla_NY)\rangle\\
&\stackrel{\textnormal{on }\Gamma}{=}&N\langle X,\mathbf G_1 Y\rangle - \langle\nabla_N X,Y\rangle - \langle X, \nabla_NY\rangle\\
&=& N\langle X, Y\rangle'_1 - \langle\nabla_N X,Y\rangle - \langle X, \nabla_NY\rangle\\
&=&\langle \nabla^{1,}{}'_N X,Y\rangle'_1+\langle X,\nabla^{1,}{}'_N Y\rangle'_1- \langle\nabla_N X,Y\rangle - \langle X, \nabla_NY\rangle\\
&=& 2L_0(X,Y)-2L_1(X,Y) \\
&=& \langle X, 2\mathbf LY\rangle
\end{eqnarray*}
where we used that in our coordinates the second fundamental forms of $\Gamma$ with respect to $N$ are $-\langle X,\nabla_N Y\rangle$ and $\langle X,\nabla^{1,}{}'_N Y\rangle^{1,}{}'$
(cf. Remark \ref{gprops}). Furthermore, recall that by construction   we have 
$\mathbf G_1N=N$ (cf. Lemma \ref{strichmetrik})  and therefore
\[(\nabla_N\mathbf G_1)N=\nabla_N(\mathbf G_1N)-\mathbf G_1(\nabla_N N)=0= \mathbf L N\]
Finally, $\nabla_N\mathbf G_1=2\mathbf L$ implies $\nabla_X\nabla_N\mathbf G_1=2\nabla_X\mathbf L$ on $\Gamma$.

We can now repeat the computations from the previous section, where the only difference occurs due to the $\nabla_N\nabla_N\mathbf G_1$ term.
\end{proof}

\section{Estimating $\mathcal R_\delta$ on $M_0$}\label{riem est section}
The goal of this section is to show that $\mc R_\delta\geq \bigl(\kappa-\varepsilon(\delta)\bigr)\mc I_\delta$
holds on $M_0$.
\begin{lemma}\label{lemma92}
We have 
\begin{eqnarray}\label{92}
\mc R - f_\delta^2\mc A + f_\delta \mc B
&\geq& (\kappa\mc-\varepsilon(\delta))\mc I+2f_\delta\bigl(-\mc L^2
                           +\frac 12 \nabla^2_N \mc G_1\bigr)                                            
\end{eqnarray}
where $\varepsilon(\delta)$  tends to zero as $\delta\rightarrow 0$. 
\begin{proof}
%Off a fixed neighborhood of $\Gamma$, i.e. for $f_\delta=0$, \eqref{92} holds without $\varepsilon$ due to the assumption on $\mc R$.  
Since $\Gamma$ is compact it suffices to show that 
\[
\mc R(\alpha,\alpha) - f_\delta^2\mc A(\alpha,\alpha) + f_\delta \mc B(\alpha,\alpha)
\geq \kappa \mc I(\alpha,\alpha) +2f_\delta\bigl(-\mc L^2
                           +\frac 12 \nabla^2_N \mc G_1\bigr)(\alpha,\alpha)-\varepsilon(\delta) \mc I(\alpha,\alpha)
\]
holds on a small neighborhood $U$ of a point $p\in\Gamma$ for every 2-form $\alpha$ on $U$ where $\varepsilon(\delta)$ does not depend on $\alpha$. 
Let us fix a coordinate neighborhood $(U,\varphi)$ of $p$ as in Section 2. 
%Since the operators in  \eqref{92} are smooth on $M_0$ we may assume
%w.l.o.g.  that their coordinate functions with respect to $\varphi=(x^1,\dots,x^n)$ are uniformly 
%bounded on $U$.
W.l.o.g. we assume that $\alpha$ has fixed 
coefficients satisfying $\sum_{i,j=1}^n(\alpha^{ij})^2=1$. 
%\\
%\\

We proceed as in Lemma 9.2 of \cite{k}. Off a $\delta$-neighborhood of  $\Gamma$, i.e. $f_\delta(x^n)=0$ the inequality holds
without an error term. For $f_\delta(x^n)=1$, i.e. on $\Gamma$ the inequality follows from 
 \eqref{onGamma} and the assumption $\mc R_1\geq \kappa$.
\\
\\
Let us now fix a point $\hat x=(x^1,\dots,x^{n-1})\in U\cap \Gamma$ and look at the inequality on  the line segment $\{(\hat x, x^n): x^n\in[0,\delta]\}$. Let 
\[
\mc Q=-\mc L^2+\frac 12 \nabla^2_N \mc G_1
\]
                           
For $x^n\in [0,\delta^4]$ we have $f_\delta(x^n)\in [0,1]$ (cf. the definition of $f_\delta$). 
If the quantities $\mc R(\alpha,\alpha)$, $\mc A(\alpha,\alpha)$, $\mc B(\alpha,\alpha)$ and $\mc Q(\alpha,\alpha)$ would not depend $x^n$, the inequality
\begin{eqnarray}\label{ineq}
\mathcal R(\alpha,\alpha) - f_\delta^2 \mc A(\alpha,\alpha) + f_\delta \mc B(\alpha,\alpha)
\geq \kappa\mc I(\alpha,\alpha)+4f_\delta \mc Q(\alpha,\alpha)
\end{eqnarray}
would hold 
without an error term since it holds for $f_\delta =0$ and $f_\delta=1$ and the function
\begin{eqnarray*}
[0,1]&\rightarrow &\mathbb R\\
y &\mapsto&  \mathcal R(\alpha,\alpha) - y^2 \mc A(\alpha,\alpha) + y\mc B(\alpha,\alpha)
\end{eqnarray*}
is concave (note that $\mb L\geq 0$ implies $\mc A=\mathbf L\we \mathbf L\geq 0$ by Lemma \ref{produkt geq 0}). However $\mathcal R(\alpha,\alpha)$, $\mc A(\alpha,\alpha)$,
 $\mc B(\alpha, \alpha)$ and $\mc Q(\alpha,\alpha)$ do
depend on $x^n$, but they are smooth on $M_0$ and hence almost constant for small $x^n$.
%Therefore, if we choose $\delta$ small enough, \eqref{ineq} still holds up to $\varepsilon(\delta)$ on the right
%hand side.
Indeed, one has for instance
\begin{eqnarray*}
|\mc R(\alpha,\alpha)(\hat x,s)-\mc R(\alpha,\alpha)(\hat x,t)|
&=&\frac 14|R_{ijkl}(\hat x,s)-R_{ijkl}(\hat x,t)||\alpha^{ij}\alpha^{kl}|\\
&\leq& 
\delta c(n)\sup_{i,j,k,l}\|R_{ijkl}\|_{C^1(U)}\\
\end{eqnarray*}
for all $s,t\in[0,\delta]$, which tends to zero 
since the $C^1$-norm of the coordinate functions is bounded if we choose $U$ small enough. Therefore \eqref {ineq} holds up to a small error term
$\ve(\delta)$ on the right hand side for $x^n\in[0,\delta^4]$.
\\
\\
For $x^n\in [\delta^4,\delta]$ we have $f_\delta(x^n)\in[-\delta^2,0]$. $\mc A,\mc B,\mc I$ and $\mc Q$ are uniformly bounded
near $\Gamma$, therefore  \eqref{ineq} holds for all $x^n\in[0,\delta]$ if we choose $\delta$ sufficiently small and subtract
another $\varepsilon(\delta)$ on the right hand side.
%\\
%\\
%Given an arbitrary $\alpha\neq 0$ in some point $x\in U$ we apply \eqref{ineq} to $ \bigl(\sum_{i,j=1}^n(\alpha^{ij}(x))^2\bigr)^{-\frac 12}\alpha(x)$ 
%and obtain
%\[
%\mathcal R(\alpha,\alpha) - f_\delta^2 \mc A(\alpha,\alpha) + f_\delta \mc B(\alpha,\alpha)
%\geq \mc I(\alpha,\alpha)+2f_\delta \mc Q(\alpha,\alpha)-\varepsilon(\delta)\sum_{i,j=1}^n(\alpha^{ij}(x))^2
%\]
%By the equivalence of metrics we have $\sum_{i,j=1}^n(\alpha^{ij}(x))^2\leq c\mc I(x)(\alpha(x),\alpha(x))$
%where $c$ depends only on $\max_{i,j}\|g_{ij}\|_{C^0(U)}$, so that \eqref{92} holds on $U$ if we replace $\varepsilon(\delta)$ with $c\varepsilon(\delta)$.
%Due to the compactness of $\Gamma$ the result follows.

\end {proof}
\end{lemma}

\begin{prop}[cf. \cite{k}, Lemma 10.1]\label{finest}
If the constant $C$ in the definition of $g_\delta$ is chosen large enough, then for small $\delta>0$ 
\[\mathcal R_\delta\geq \bigl(\kappa -\varepsilon(\delta)\bigr)\mathcal I_\delta\]
where $\varepsilon(\delta)\rightarrow 0$ as $\delta$ tends to zero.

\begin{proof}
Since $g_\delta\rightarrow g $ in the $C^0$-sense, it suffices to show 
\[\mathcal R_\delta\geq \bigl(\kappa -\varepsilon(\delta)\bigr)\mathcal I\]
From Proposition \ref{operatorzerlegung} and Lemma \ref{lemma92} we get 
\begin{eqnarray*}
\mc R_\delta
&\approx& 
\mc R - f^2_\delta \mc A + f_\delta\mc B
 - 2f'_\delta\mc L + 2f^2_\delta \mc L^2 + 2Cf_\delta \hat{\mc I}\\
&\geq& \kappa \mc I
+2f_\delta \bigl( -\mc L^2+\frac 12\nabla_N^2\mc G_1+ C\hat{\mc I}\bigr)                                            
-2f'_\delta\mc L + 2f^2_\delta \mc L^2-\varepsilon(\delta)\mc I
\end{eqnarray*}
By definition we have 
\begin{eqnarray}\label{LG} 
-\mc L^2+\frac 12\nabla_N^2\mc G_1+ C\hat{\mc I}=(-\mathbf L^2+\frac 12\nabla_N^2\mathbf G_1+C\mathbf P^T)\we \mathbf P^N
\end{eqnarray}
Note that the operators $\mathbf L^2$ and $\nabla_N^2\mathbf G_1$ vanish on $T\Gamma(d)^\perp$. Therefore  \eqref{LG} becomes 
nonnegative near $\Gamma$ for a large enough $C$ in view of Lemma \ref{produkt geq 0} (recall that $\mathbf P^N$ is nonnegative). 
Moreover, $-\mc L^2+\frac 12\nabla_N^2\mc G_1+ C\hat{\mc I}$
is uniformly bounded near $\Gamma$, and $f_\delta\geq-\delta^2$ by definition. Therefore
\[2f_\delta \bigl( -\mc L^2+\frac 12\nabla_N^2\mc G_1+ C\hat{\mc I}\bigr) \geq -\varepsilon(\delta)\mc I\]
$f_\delta '$ is negative on $[0,\delta^4)$, does not exceed $\varepsilon(\delta)$ on $[\delta^4,\delta]$ and vanishes else. 
$\mc L=\mathbf L\we \mathbf P^N$ is nonnegative 
 and uniformly bounded near $\Gamma$, which gives us
$-4f'_\delta\mc L\geq -\varepsilon(\delta)\mc I$.
Obviously $f_\delta^2 \mc L^2=f_\delta^2\mathbf L^2\we \mathbf P^N$ is nonnegative, and we are done.
%Off a fixed $\delta$-neighbourhood of $\Gamma$ the claim follows due to the definition of $f_\delta$.
%In some fixed point $x$ near $\Gamma$ we may choose an orthonormal frame such that $g_{ij}=\delta_{ij}$ and 
%$(\nabla^2 G_1)_{ij}=\lambda_i\delta_{ij}$. Given a 2-form $\alpha$ we compute
%\begin{eqnarray*}
%\frac 12\nabla_N^2\mc G_1(\alpha,\alpha)+ C\hat{\mc I}(\alpha,\alpha)&=&
%\sum_{j,l=1}^n \frac 12(\nabla^2 G_1)_{jl}\alpha^{nj}\alpha^{nl}+\sum_{j,l=1}^n Cg_{jl}\alpha^{nj}\alpha^{nl}\\
%&=&\sum_{j,l=1}^n (\frac 12 \lambda_j+C)(\alpha^{nj})^2\geq 0
%\end{eqnarray*}
%for large enough $C$ depending only on the $C^0$-norm of the eigenvalues of $\nabla^2G_1$ near $\Gamma$, which is bouded since $\Gamma$ is compact.
%Thus $\mc L^2+\frac 12\nabla_N^2\mc G_1+ C\hat{\mc I}$ becomes nonnegative for a large enough fixed $C$. Moreover we have 
%$f_{\delta}\geq -\delta^2$, so we can estimate the second term on the right hand side of the inequality from below by 
%$-\varepsilon(\delta)\mc I$.
%$f_\delta '$ is negative on $[0,\delta^4)$, does not exceed $\delta$ on $[\delta^4,\delta]$ and vanishes else. 
%$\mc L$ is nonnegative and uniformly bounded near $\Gamma$, which gives us
%$-4f'_\delta\mc L\geq -\varepsilon(\delta)\mc I$.
%Obviously $f_\delta^2 \mc L^2$ is nonnegative, and we are done.
\end{proof}
\end{prop}

\begin{cor}\label{geqae}
The weakly defined Riemannian curvature operator of the $W^{2,\infty}_{loc}$-metric $g_{(\delta)}$ on $M$, (recall that $g_{(\delta)}|_{M_0}=g_\delta$ and $g_{(\delta)}|_{M_1}=g_1$)
satisfies $\mc R(g_{(\delta)})\geq \kappa-\varepsilon(\delta)$ a.e.
(everywhere except on $\Gamma$).
\begin{proof}
In local coordinates the Riemannian curvature tensor of some metric $h$ is given by
\begin{eqnarray}\label{curvform}
R(h)_{ijkl}=\partial_j\partial_kh_{il}+\partial_i\partial_lh_{jk}-\partial_j\partial_lh_{ik}-\partial_i\partial_kh_{jl}
+(h^{-1}\bullet\partial h\bullet \partial h)_{ijkl}
\end{eqnarray}
where $\bullet$ means contracting tensors using the metric.
Since the second derivatives enter \eqref{curvform} linearly, $\mc R(g_{(\delta)})$ can be defined on $M$ in the weak sense. 
$\mc R(g_{(\delta)})\geq \kappa-\varepsilon(\delta)$ a.e. follows from Proposition \ref{finest} and the assumption $\mc R(g_1)\geq \kappa$ .

\end{proof}
\end{cor}

\section {Mollifying $g_{(\delta)}$}
By mollifying $g_{(\delta)}$ we construct a family of smooth metrics with properties as required in Definition \ref{c0curv}.

\begin{prop}
There exists a family of smooth metrics $\tilde g_{(\delta)}$ such that 
\[\tilde g_{(\delta)}\rightarrow g \quad\textnormal{uniformly on compact subsets of }M\]
and 
\[\tilde {\mathcal R}_{(\delta)}\geq (\kappa-\tilde\varepsilon(\delta)) \tilde{\mathcal I}_{(\delta)}\]
holds with $\tilde{\varepsilon}(\delta)\rightarrow 0$.

\begin{proof}
Let us fix a small $\delta>0$.
We choose a locally finite cover of coordinate neighborhoods $(U_s)$ such that $U_s\subset \subset U_s'$ for some coordinate chart 
$U'_s$. Since $\Gamma $ is compact, we may assume w.l.o.g. that $U'_s\cap \Gamma=\emptyset$ for $s> N$ for some $N\in \mathbb N$.
We denote the coordinate functions of $g_{(\delta)}$ on $U'_s$ by 
$(g^s_{(\delta)})_{ij}$. After choosing $U_s'$ even smaller if necessary we may also assume that $\|(g^s_{(\delta)})_{ij}\|_{C^1(U_s')}\leq C<\infty$ for all $s\leq N$.
 For $s\leq N$ and $x\in U_s$ let 
\begin{eqnarray}\label{mol1}
(g^{s,h}_{(\delta)})_{ij}(x)=(\rho_h\ast(g^s_{(\delta)})_{ij})(x)=\int_{|z|\leq 1}\rho(z)(g^s_{(\delta)})_{ij}(x-hz)dz
\end{eqnarray}
where $\rho\in C^\infty_0(\mathbb R^n)$ satisfies $\textnormal{supp}\,\rho\subset B_1(0)$ and $\int_{\mathbb R^n}\rho=1$, and 
$h$ is small enough so that for all $s\leq N$ $x-hz$ lies in $U'_s$ for all $z\in B_1(0)$ (here we identified the coordinate neighborhoods on $M$ with the corresponding
 neighborhoods on $\mathbb R^n$).
$g_{(\delta)}^{s,h}$ is a well defined metric on $U_s$ which converges to $g_{(\delta)}|_{U_s}$ in the $C^1$-sense. 
  Let $(\eta_s)$ be a partition of unity on $M$ such that $\textnormal{supp}\,\eta_s\subset U_s$ for all $s$.
  For $h$ as above we then define a smooth metric $g^h_{(\delta)}$ on $M$ by 
\begin{eqnarray}\label{mol2}
g^h_{(\delta)}=\sum_{s\leq N} \eta_s g^{s,h}_{(\delta)}+\sum_{s>N}\eta_s g_{(\delta)}
\end{eqnarray}
%(Note that we mollify $g_{(\delta)}$ only in a small neighborhood of $\Gamma$, namely on $\bigcup_{s\leq N}U_s$, this way 
%we will be able to choose $\tilde\varepsilon(\delta)$ independent from the coordinate neighborhood.) 
We now calculate the Riemannian curvature tensor $R(g^h_{(\delta)})$ using the formula \eqref{curvform}. The terms which do not involve any derivatives of the unity functions $\eta_s$ give us just the mollified Riemannian curvature tensor $(R(g_{(\delta)}))^h$, constructed in the same way as $g_{(\delta)}^h$ in \eqref{mol1} and \eqref{mol2}.
The other terms vanish uniformly on $M$ as $h$ tends to zero. We shall verify this exemplary for one of them. After fixing a coordinate chart $(U,\varphi)$ we compute
\begin{eqnarray*}
&&|\sum_{s\leq N} \partial_j\partial_k\eta_s (g^{s,h}_{(\delta)})_{il}+\sum_{s>N}\partial_j\partial_k\eta_s (g_{(\delta)})_{il}|\\
&=&|\sum_{s\leq N} \partial_j\partial_k\eta_s (g_{(\delta)})_{il}+
\sum_{s\leq N} \partial_j\partial_k\eta_s \bigl((g^{s,h}_{(\delta)})_{il} - (g_{(\delta)})_{il}\bigr)
+\sum_{s>N}\partial_j\partial_k\eta_s (g_{(\delta)})_{il}|\\
&\leq&|\partial_j\partial_k\underbrace{(\sum_{s\geq 1}\eta_s)}_{\equiv 1}(g_{(\delta)})_{il} 
+ \sum_{s\leq N} |\partial_j\partial_k\eta_s| |(g^{s,h}_{(\delta)})_{il} - (g_{(\delta)})_{il}|\\
&\leq&
N\bigl(\max_{s=1,\dots,N}\|\eta_s\|_{C^2(U_s)}\bigr)
\bigl(\max_{s=1,\dots,N}\max_{i,l=1,\dots,n}\|(g^{s,h}_{(\delta)})_{il} - (g_{(\delta)})_{il}\||_{C^0(U_s)}\bigr)\\
&\stackrel{h\rightarrow 0}{\rightarrow}& 0
\end{eqnarray*}
All in all we obtain 
\[|(R(g^h_{(\delta)}))_{ijkl} - (R(g_{(\delta)}))^h)_{ijkl}|\leq \varepsilon(\delta,h)\]
where $\ve(\delta,h)\stackrel{h\rightarrow 0}{\rightarrow} 0$ for every fixed $\delta$, 
which implies that
\begin{eqnarray}\label{ungl2}
\mathcal R(g^h_{(\delta)})\geq(\mathcal R(g_{(\delta)}))^h   -\tilde\varepsilon(\delta,h)\mathcal I(g_{(\delta)})
\end{eqnarray}
Moreover, Corollary \ref{geqae} implies
\begin{eqnarray}\label{ungl3}
(\mathcal R(g_{(\delta)}))^h\geq(\kappa-\varepsilon(\delta))(\mathcal I(g_{(\delta)}))^h
\end{eqnarray}
holds. Indeed, for any two form $\alpha$ on $U_{s'}$ (w.l.o.g. with fixed coefficients) we have 
\begin{eqnarray*}
(\mc R(g_{(\delta)}))^{s,h}(x)(\alpha,\alpha)&=&
\int_{|z|\leq 1}\rho(z)\frac 14 (R(g_{(\delta)}))^s_{ijkl}(x-hz)\alpha^{ij}\alpha^{kl}dz\\
&\stackrel{\ref{geqae}}{\geq}&\bigl(\kappa-\varepsilon(\delta)\bigr)\int_{|z|\leq 1}\rho(z) \frac 14(\mc I(g_{(\delta)}))^s_{ijkl}(x-hz)\alpha^{ij}\alpha^{kl}dz\\
&=&\bigl(\kappa-\varepsilon(\delta)\bigr) \mc (\mc I(g_{(\delta)}))^{s,h}(x)(\alpha,\alpha)
\end{eqnarray*}
Combining \eqref{ungl2} and \eqref{ungl3} we arrive at
\begin{eqnarray*}
\mathcal R(g^h_{(\delta)})&\geq&
(\kappa-\varepsilon(\delta))(\mathcal I(g_{(\delta)}))^h   -\tilde\varepsilon(\delta,h)\mathcal I(g_{(\delta)})\\
&\geq&(\kappa-\varepsilon(\delta))(1\pm\ve(\delta))\mathcal I(g_{(\delta)}^h)   -\tilde\varepsilon(\delta,h)(1+\ve(\delta))\mathcal I(g_{(\delta)}^h)
\end{eqnarray*}
where we used the fact that for every fixed $\delta$ 
both $(\mathcal I(g_{(\delta)}))^h$ and $\mathcal I(g^h_{(\delta)})$ approach $\mathcal I(g_{(\delta)})$
 as $h$ tends to zero ($\pm$ referes to $\kappa\geq 0,\kappa<0$, respectively). 
 Since $\tilde\ve(\delta,h)\rightarrow 0$ as $h\rightarrow 0$ for every fixed $\delta$, we may choose $h$ small enough such that $\tilde\ve(\delta,h)\leq \ve(\delta)$,
 thereby obtaining
\begin{eqnarray*}
\mathcal R(g^h_{(\delta)})\geq\bigl(\kappa-(|\kappa|+3)\ve(\delta)\bigr)\mc I(g_{(\delta)}^h)
\end{eqnarray*}
and the desired result follows  with
$\tilde g_{(\delta)}=g^h_{(\delta)}$ and $\tilde\varepsilon(\delta)=(|\kappa|+3)\varepsilon(\delta)$.
\end{proof}
\end{prop}

\section{A similar result for other operators}
As mentioned in the introduction, an analogue result can be shown for manifolds with lower bounds on the Ricci curvature, scalar curvature, bi-curvature,
isotropic curvature and flag curvature, respectively.
\subsection{Manifolds with Ricci curvature $\geq \kappa$}\label{ricci section}

\begin{thm}
Let $(M_0,g_0)$, $(M_1,g_1)$, $(M,g)$ and $L=L_0+L_1$ be as in Theorem \ref{main}. 
Suppose that $\Ric(g_0)$ and $\Ric(g_1)$ are at least $\kappa$ (in the sense of eigenvalues). If
$L$ is positive semidefinite, then $\Ric(g)$ is at least $\kappa$ (in a similar sense as in Definition \ref{c0curv}). 
%The analogue result holds if 
%the scalar curvatures of $g_0$ and $g_1$ are at least $\kappa$. 
\begin{proof}

Given a symmetric bilinear form $\mc T$ on $\Lambda^2(TM)$ and a metric $h$ we denote 
\[\Ric_h(\mc T)=h^{jl}T(\cdot,\partial_j,\cdot,\partial_l)\]
where $T(\partial_i,\partial_j,\partial_k,\partial_l)=\mc T(\partial_i\we\partial_j,\partial_k\we\partial_l)$.
The strategy of the proof is similar to that of the proof of Theorem \ref{main}. We show 
\begin{itemize}
\item[(a)]
The curvature operator of the modified metric $g_\delta$ on $M_0$ satisfies $\Ric_{g_\delta}(\mc R_\delta)\geq \bigl(\kappa-\varepsilon(\delta)\bigr)g_\delta$ 
with $\varepsilon(\delta)\rightarrow 0$ (this corresponds with Lemma \ref{finest})\\
\item[(b)] By mollifying $g_{(\delta)}$ we construct a family of smooth metrics which approximate $g$ in the $C^0$-sense and have Ricci curvature at least
$\kappa -\varepsilon(\delta)$.
\end{itemize}
\noindent
%The proof in the scalar curvature case is similar.
%\\
%\\
(a):
Here we may simplify the argument of the previous sections. Recall that we identify endomorphisms and bilinear forms on $TM_0$ in the sense of Notation \ref{optens}.
In view of this identification, we have $g=\id_{TM_0}$. Since $g_\delta\approx g$ on $M_0$, it suffices to show 
$\Ric_{g_\delta}(\mc R_\delta)\geq\bigl(\kappa-\ve(\delta)\bigr)\id_{TM_0}$.
By \eqref{RdeltaOp} we have 
\begin{eqnarray*}
\Ric_{g_\delta}(\mc R_\delta) &\geq& \Ric_{g_\delta}(\mc R)-f_\delta^2\Ric_{g_\delta}(\mc A) + f_\delta\Ric_{g_\delta}( \mc B)\\
&-& 2f_\delta'\Ric_{g_\delta}(\mc L) +2f_\delta^2\Ric_{g_\delta}(\mc L^2) + 2 C f_\delta \Ric_{g_\delta}(  \hat {\mc I})-\varepsilon(\delta) \id_{TM_0}\nonumber
\end{eqnarray*}
Since $f_{\delta}$ is bounded and $g_\delta\rightarrow g$ in the $C^0$-sense, we may replace $\Ric_{g_\delta}$ by $\Ric_g$ everywhere except in the $f_\delta '$ term,
i.e. we have
\begin{eqnarray}\label{ric}
\Ric_{g_\delta}(\mc R_\delta) &\geq& \Ric_{g}(\mc R)-f_\delta^2\Ric_{g}(\mc A) + f_\delta\Ric_{g}( \mc B)\\
&-& 2f_\delta'\Ric_{g_\delta}(\mc L) +2f_\delta^2\Ric_{g}(\mc L^2) + 2 C f_\delta \Ric_{g}(  \hat {\mc I})-\varepsilon(\delta) \id_{TM_0}\nonumber
\end{eqnarray}
Recall that $\hat{\mc I}=\mb P^T\we\mb P^N$ (cf. Notation \ref{proj op}). We compute
\begin{eqnarray}\label{ric4}
\bigl(\Ric_g(\mb P^T\we\mb P^N)\bigr)_{ik}&=&\frac 12g^{jl}(P^T_{ik}P^N_{jl}-P^T_{jk}P^N_{il}+P^N_{ik}P^T_{jl}-P^N_{jk}P^T_{il})\nonumber\\
&=&
\frac 12(\tr_g (\mb P^N) P^T_{jk}+\tr_g(\mb P^T)P^N_{jk})\nonumber\\
&=&
\frac 12(P^T_{ik} + (n-1)P^N_{ik})
\end{eqnarray}
If we assume that $n\geq 2$ (the case $n=1$ is trivial), this implies
\begin{eqnarray}
\Ric_g(\hat{\mc I})\geq \frac 12 (\mb P^T+\mb P^N) = \frac 12\id_{TM_0}
\end{eqnarray}
Therefore, using the assumption $\Ric_g(\mc R)\geq \kappa$, we can estimate the right hand side of \eqref{ric} from below by
\begin{eqnarray}\label{ric2}
\Ric_{g_\delta}(\mc R_\delta) &\geq&( \kappa-\ve(\delta))\id_{TM_0}-f_\delta^2\Ric_{g}(\mc A) + f_\delta\Ric_{g}( \mc B)\nonumber\\
&-& 2f_\delta'\Ric_{g_\delta}(\mc L) +2f_\delta^2\Ric_{g}(\mc L^2) +  C f_\delta \id_{TM_0}\nonumber\\
&=&( \kappa-\ve(\delta))\id_{TM_0}
- 2f_\delta'\Ric_{g_\delta}(\mc L) \nonumber\\
&+& f_\delta \bigl(-f_{\delta}\Ric_{g}(\mc A) + \Ric_{g}( \mc B)+2f_\delta\Ric_{g}(\mc L^2) +  C \id_{TM_0}\bigr )
\end{eqnarray}
The operators $\mc A$, $\mc B$ and $\mc L^2$ are smooth and hence uniformly bounded near $\Gamma$. Therefore, the term in parenthesis in \eqref{ric2} 
is 
nonnegative for large enough fixed $C$ and bounded from above\footnote{
Note that at this point we simplified the argument of Section \ref{riem est section}. $\Ric_g(\mb P^T\we\mb P^N)$ is estimated from below 
by the positive definite operator $\frac 12\id_{TM_0}$, 
hence the $\mc A,\mc B$ and $\mc L^2$ terms are absorbed by $C\id_{TM_0}$. When considering the full curvature tensor, the corresponding operator
$C\mb P^T\we\mb P^N$ has nontrivial kernel, which is why we needed the concavity argument of Lemma \ref{lemma92}.

}.
 Since $f_{\delta}\in [-\delta^2, 1]$, the last expression in \eqref{ric2} is $\geq-\ve(\delta)\id_{TM_0}$,
and we arrive at
\begin{eqnarray}\label{ric3}
\Ric_{g_\delta}(\mc R_\delta) &\geq&( \kappa-\ve(\delta))\id_{TM_0}
- 2f_\delta'\Ric_{g_\delta}(\mc L)
\end{eqnarray}
Finally, we compute the $f_\delta'$-term in \eqref {ric3}. Let us fix a point $x\in M_0$ near $\Gamma$.
In the construction of local coordinates in Section 2 we may additionally choose $x^1,\dots,x^{n-1}$ such that $\partial_1(x),\dots,\partial_{n-1}(x)$
are orthonormal with respect to $g(x)$ and $L(x)$ is diagonal. By construction this implies that $g_\delta(x)$ is diagonal, 
$(g_\delta)_{jl}(x)=\mu_l\delta_{jl}$, where $\mu_l >0$ since $g_\delta$ is positive definite. Moreover, we still have 
$(P^N)_{ij}=\delta_{in}\delta_{jn}$ in these coordinates. Therefore, given a vector $\xi$, in $x$ we compute using $L_{kn}=0$ for $k=1,\dots,n$:

\begin{eqnarray*}
(\Ric_{g_\delta}(\mc L))(\xi,\xi)&=&
g_\delta^{jl}(\mathbf L\we\mathbf P^N)_{ijkl}\xi^i\xi^k =\sum_{l=1}^n \frac {1}{\mu_l}(\mathbf L\we\mathbf P^N)_{ilkl}\xi^i\xi^k \\
&=&\frac 12 \sum_{l=1}^n \frac{1}{\mu_l}(L_{ik}P^N_{ll} - L_{lk}P^N_{il} +  P^N_{ik}L_{ll}   - P^N_{lk}L_{il})\xi^i\xi^k\\
&=&\frac 12\frac {1}{\mu_n} L(\xi,\xi) + \frac 12 (\xi^n)^2\sum_{l=1}^n\frac{1}{\mu_l}L_{ll}\geq 0
\end{eqnarray*}
since $L\geq 0$ by assumption.
We then proceed as in the proof of Proposition \ref{finest}
and estimate the $f_\delta'$-term from below by $-\varepsilon(\delta)\id_{TM_0}$.
\\
\\
(b): Let us fix a $\delta>0$. We construct the metrics $g^h_{(\delta)}$ as in Section 6. By \eqref{ungl2} and since $g^h_{(\delta)}\rightarrow g_{(\delta)}$
uniformly as $h\rightarrow 0$ we have
\[
\Ric_{g^h_{(\delta)}}(\mc R(g^h_{(\delta)}))
\geq \Ric_{g_{(\delta)}}(\mathcal R(g_{(\delta)}))^h   -\tilde\varepsilon(\delta,h)(g_{(\delta)})
\]
where $\tilde\ve(\delta,h)\rightarrow 0$ as $h\rightarrow 0$.
Given a vectorfield $X$ on $U_{s'}$ which has constant coefficients not exceeding 1, on $U_s$ we compute 
using (a) and the mean value theorem
\begin{eqnarray*}
\Ric_{g_{(\delta)}}(\mathcal R(g_{(\delta)}))^{h,s}(x)(X,X)&=&
\int_{|z|\leq 1}\rho(z) (g_{(\delta)})^{jl}(x)(R(g_{(\delta)}))^s_{ijkl}(x-hz)X^iX^k dz\\
&=&\int_{|z|\leq 1}\rho(z) (g_{(\delta)})^{jl}(x-hz)(R(g_{(\delta)}))^s_{ijkl}(x-hz)X^iX^k dz\\
&+&h\int_{|z|\leq 1}\rho(z) D(g_{(\delta)})^{jl}(\xi_{x,hz})z(R(g_{(\delta)}))^s_{ijkl}(x-hz)X^iX^k dz\\
&\geq& \bigl(\kappa-\varepsilon(\delta)\bigr)g_{(\delta)}^{s,h}(X,X)-hC(\delta)\\
&\geq& \bigl(\kappa-2\varepsilon(\delta)\bigr)g_{(\delta)}^{s,h}(X,X)
\end{eqnarray*}
where $\xi_{x,hz}=(1-t)x+thz$ for some $t\in[0,1]$, and $C(\delta)$ depends on the bound of $\mc R(g_{(\delta)})$ near $\Gamma$, which is finite for every fixed $\delta$.
Note that for every fixed $\delta$ we may choose $h$ small enough so that $hC(\delta)\leq \ve(\delta)$. Since $U_{s}\cap \Gamma\neq\emptyset$ only for finitely many 
$s$, we deduce

\[
\Ric_{g_{(\delta)}}(\mathcal R(g_{(\delta)}))^{h}
\geq \bigl(\kappa-2\varepsilon(\delta)\bigr)g_{(\delta)}^{h}\]
Thus
\[
\Ric_{g^h_{(\delta)}}(\mc R(g^h_{(\delta)}))
\geq \bigl(\kappa-2\varepsilon(\delta)\bigr)g_{(\delta)}^{h}   -\tilde\varepsilon(\delta,h)(g_{(\delta)})
\]
Finally we choose $h$ even smaller such that $\tilde\ve(\delta,h)\leq\ve(\delta)$ and $g_{(\delta)}\leq(1+\ve(\delta))   g_{(\delta)}^{h}$,
and the result follows with $\tilde g_{(\delta)}=g_{(\delta)}^{h}$ and $\tilde \ve(\delta)=4\ve(\delta)$.

\end{proof}
\end{thm}

\subsection{Manifolds with scalar curvature $\geq \kappa$}\label{scalar section}
The scalar curvature of a $C^2$-smooth Riemannian metric $g$ is defined as $\sca(g)=\tr_g\Ric_g=g^{ik}g^{jl}R^g_{ijkl}$.
As mentioned in the introduction, in the scalar curvature case we may weaken the assumption $L\geq 0$ on $\Gamma$
to $\tr_gL\geq 0$ on $\Gamma$, i.e. the sum of the mean curvatures of $g_0$ and $g_0$ is nonnegative. 
\begin{thm}\label{scalar case}
Let $(M_0,g_0)$, $(M_1,g_1)$, $(M,g)$ and $L=L_0+L_1$ be as in Theorem \ref{main}. 
Suppose that $\sca(g_0)$ and $\sca(g_1)$ are at least $\kappa$. 
If $\tr_g L\geq 0$ on $\Gamma$, then $\sca(g)\geq \kappa$,(in a similar sense as in Definition \ref{c0curv}).

\end{thm}
\begin{proof}
First let us assume $\tr_gL>0$ on $\Gamma$. In analogy to Lemma \ref{defnL}, we need to
verify that the extension of $L$ satisfies $\tr_g L>0$, if so does the initial operator on $\Gamma$. In fact, for $x\in M_0$ near $\Gamma$ 
we have $\tr_{g(x)}L(x)=\tr_{g(\hat x)}L(\hat x)$, where $\hat x$ is the point of $\Gamma$ nearest to $x$. 
Indeed, let $x\in M_0$ be a point near $\Gamma$ such that the extention $\mb L$ is well defined in $x$.
%, and let $\hat x\in \Gamma$ be be the point nearest to $x$. 
Recall that 
for $X\in T_xM_0$ we defined $\mb L X=P^{-1}\mb L PX$, where $P$ is the parallel transportation along the integral curves of the normal field $N$, which takes $X\in T_xM_0$ to 
$PX\in T_{\hat x}M_0$. Let $e_1,\dots e_n$ be an orthonormal basis of $T_xM_0$, and let $g_{ij}(x)=\la e_i,e_j\ra_{g(x)}=\delta_{ij}$ and $L_{ij}(x)=\la\mb L (x)e_i,e_j\ra_{g(x)}$.
We compute
\begin{eqnarray}\label{trace L}
\tr_{g(x)} L(x)&=&
g^{ij}(x)L_{ij}(x) =\sum_{i=1}^n\la \mb L(x) e_i,e_j\ra_{g(x)}\nonumber\\
&=&\sum_{i=1}^n\la P^{-1}\mb L(\hat x) Pe_i,e_j\ra_{g(x)}
=\sum_{i=1}^n\la \mb L(\hat x) Pe_i,P e_j\ra_{g(\hat x)}\nonumber\\
&=& \sum_{i=1}^nL(\hat x)(Pe_i, Pe_i) =\tr_{g_(\hat x)} L(\hat x)
\end{eqnarray}
since $Pe_1,\dots, Pe_n$ is an orthonormal basis of $T_{\hat x}M_0$.
\\

Given a metric $h$ and a bilinear form $\mc T\in \Lambda^2(TM)$ we denote:
\[\sca_h(\mc T)=h^{ik}h^{kl} T_{ijkl}\]
where $T_{ijkl}=\mc T(\partial_i\we\partial_j,\partial_k\we\partial_l)$.
%Since the proof is similar as in the Ricci curvature case, we only repeat the crucial steps.
By \eqref{RdeltaOp} we have 
\begin{eqnarray}\label{scalar ineq}
\sca_{g_\delta}(\mc R_\delta) &\geq& \sca_{g_\delta}(\mc R)-f_\delta^2\sca_{g_\delta}(\mc A) + f_\delta\sca_{g_\delta}( \mc B)\nonumber\\
&-& 2f_\delta'\sca_{g_\delta}(\mc L) +2f_\delta^2\sca_{g_\delta}(\mc L^2) + 2 C f_\delta \sca_{g_\delta}(  \hat {\mc I})-\varepsilon(\delta) \nonumber\\
&\geq& \sca_g(\mc R)-f_\delta^2\sca_g(\mc A) + f_\delta\sca_g( \mc B)\nonumber\\
&-& 2f_\delta'\sca_{g_\delta}(\mc L) +2f_\delta^2\sca_g(\mc L^2) + 2 C f_\delta \sca_g(  \hat {\mc I})-\varepsilon(\delta)
\end{eqnarray}
where we used that $g_\delta\rightarrow g$ in the $C^0$-sense and the fact that all terms, except for the $f_\delta'$-term, remain bounded as $\delta\rightarrow 0$.
In view of \eqref{ric4} we have
\begin{eqnarray*}
\sca_g{\hat{\mc I}}=\frac 12 g^{ik}(P^T_{ik}+(n-1)P^N_{ik})=n-1>0
\end{eqnarray*}
if  $n\geq 2$. Similarly as in the previous section, we use the assumption $\sca_g(\mc R)\geq 0$ and the fact that $\mc A$, $\mc B$ and $\mc L^2$ are bounded near $\Gamma$
and $f_\delta$ is almost nonnegative, so that after choosing $C$ large enough, we may estimate \eqref{scalar ineq} from below by
\begin{eqnarray}\label{sca1}
\sca_{g_\delta}(\mc R_\delta)&\geq&\kappa-\ve(\delta)- 2f_\delta'\sca_{g_\delta}(\mc L) 
\end{eqnarray}

Consider  the $f'_\delta$-term in the above expression. As in the previous section, in some point $x\in M_0$ near $\Gamma$ we may choose local coordinates such that
$g_{ij}=\delta_{ij}$, $L_{ij}=\lambda_i\delta_{ij}$, $(g_\delta)_{ij}=\mu^\delta_i\delta_{ij}$ and $P^N_{ij}=\delta_{in}\delta_{jn}$. In these coordinates we have
(recall that $\lambda_n=L_{nn}=0$ and $\mu^\delta_n=1$)
\begin{eqnarray}\label{bed b}
(\sca_{g_\delta}(\mc L))&=&
g_\delta^{ik}g_\delta^{jl}(\mathbf L\we\mathbf P^N)_{ijkl} =\sum_{i,j=1}^n \frac {1}{\mu^\delta_i}\frac{1}{\mu^\delta_j}(\mathbf L\we\mathbf P^N)_{ijij} \nonumber\\
&=&\frac 12 \sum_{i,j=1}^n \frac{1}{\mu^\delta_i}\frac{1}{\mu^\delta_j}(L_{ii}P^N_{jj} - L_{ij}P^N_{ij} +  P^N_{ii}L_{jj}   - P^N_{ij}L_{ij})\nonumber\\
&=&\frac{1}{\mu^\delta_n}\sum_{i=1}^n \frac{1}{\mu^\delta_i}\lambda_i=\sum_{i=1}^{n-1} \frac{1}{\mu^\delta_i}\lambda_i
\end{eqnarray}
Note that the eigenvalues $\mu_i^\delta\rightarrow 1$ since $g_\delta\rightarrow g$ uniformly, hence $\tr_g(\mb L)=\sum_{i=1}^n\lambda_i>0$ implies 
\[
\sum_{i=1}^{n-1} \frac{1}{\mu^\delta_i}\lambda_i\geq (1-\ve(\delta))\sum_{i=1}^{n-1}\lambda_i\geq 0
\]
for small enough $\delta$. We then proceed as in the previous section and estimate the $f_\delta'$-term in \eqref{sca1} from below by $-\ve(\delta)$, which 
completes the proof for the case $\tr_g L>0$ on $\Gamma$.\\
\\
Let us now study the case where $\tr_gL\geq 0$ on $\Gamma$. In this case we may slightly modify either one of the initial metrics $g_0$ or $g_1$ near the boundary, such that 
$\tr_gL$ becomes strictly positive, and then repeat the argument above. More precisely, consider $g_0$ near $\Gamma$. Recall that in local coordinates $(x^1,\dots,x^n)$ we chose in Section 2, 
$g_0$ has the form  
\[g_0=
\begin{pmatrix}
  \hat g_0  & 0\\
  0& 1    
 \end{pmatrix}
 \]
where $\hat g$ is the restriction of $g$ to the equidistant hypersurfaces $\Gamma(d)$, $d=dist_g(\Gamma,\cdot)=x^n$. Let $d_0>0$ small enough so that 
$\Gamma(d)$ is smooth for $d\leq d_0$. We find a smooth function $\vf:\mathbb R_{\geq 0}\rightarrow \mathbb R_{\geq0}$ satisfying
\[\vf(0)=1,\,
\vf' (0)<0,\,
\vf|_{[d_0,\infty)}\equiv1 \, \textnormal{ and } |\vf'|,|\vf''|\leq\ve
\]
with $\ve$ small, and put 
\[\tilde g_0=
\begin{pmatrix}
  \vf(x^n)\hat g_0  & 0\\
  0& 1    
 \end{pmatrix}
 \]
 Note that in view of $\vf(0)=1$ we have $\tilde g_0|_\Gamma=g_0|_\Gamma=g_1|_\Gamma$, so that the isometry of the boundaries is preserved.
As in Lemma \ref{gprops}, in a point $p\in\Gamma$ (i.e. $x^n(p)=0$) we compute 
\begin{eqnarray*}
\tilde L^0_{ij}=-\frac 12\partial_n\tilde g^0_{ij}=-\frac 12 \vf'(0) g^0_{ij}-\frac 12\vf(0) \partial_ng^0_{ij}=-\frac 12 \vf'(0) g^0_{ij} + L^0_{ij}
\end{eqnarray*}
and thus
\begin{eqnarray*}
\tr_{\tilde g_0}(\tilde L_0)=g_0^{ij}(-\frac 12 \vf'(0) g^0_{ij} + L^0_{ij})=-\frac {n}{2}\vf'(0) +\tr_{g_0}L_0 >\tr_{g_0}L_0
\end{eqnarray*}
which gives us $\tr_{\tilde g_0}\tilde L_0+\tr_{g_1}L_1>0$, since by assumption $\tr_{g_0}L_0+\tr_{g_1}L_1=\tr_g L\geq 0$ on $\Gamma$.
Moreover, by construction, the new metric $\tilde g_0$ is $C^2$-close to $g_0$, thus their scalar curvatures differ only by an error term $\ve$ coming from the 
first two derivatives of $\vf$, which we may choose arbitrary small. We then may replace $g_0$ by $\tilde g_0$ and proceed like in the $\tr_gL>0$ case.

\end{proof}
\noindent
\emph{Remark:}
In \cite{Miao} P. Miao generalized the positive mass theorem \cite{Schoen Yau} (which says that 
an assymtotically flat manifold with nonnegative scalar curvature has nonnegative ADM mass),
 to metrics which fail to be $C^1$ across a hypersurface $\Sigma$. One of the essential steps of his proof 
was to smoothen the metric across $\Sigma$ in such a way that the scalar curvarture stays bounded from below by a constant 
(cf. \cite{Miao}, Proposition 3.1). Theorem \ref{scalar case} provides a slightly better approximation, since in our case the smooth metrics have scalar
curvature $\geq-\ve$.

\subsection{Manifolds with bi-curvature $\geq \kappa$}

The bi-curvature $\bi (g)$ of a $C^2$-smooth Riemannian metric $g$  is defined as the sum of the two smallest 
eigenvalues of $\mathcal R(g)$. Note that $\bi(g)\geq \kappa$ holds on $M$ iff
\[\mathcal R(g)(\alpha,\alpha)+\mathcal R(g)(\beta,\beta)\geq \kappa\]
for all $\alpha,\beta\in \Lambda^2(TM)$ which are orthonormal with respect to $g$.

\begin{thm}
Let $(M_0,g_0)$, $(M_1,g_1)$, $(M,g)$ and $L=L_0+L_1$ be as in Theorem \ref{main}. 
Suppose that $\bi(g_0)$ and $\bi(g_1)$ are at least $\kappa$. If
$L$ is positive semidefinite, then $\bi(g)\geq\kappa$ (in a similar sense as in Definition \ref{c0curv}).
\begin {proof}
We proceed  as in the previous section and show 
\begin{itemize}
\item[(a)]
The modified metric $g_\delta$ on $M_0$ satisfies $\bi(g_\delta)\geq \kappa-\varepsilon(\delta)$ 
where $\varepsilon(\delta)\rightarrow 0$ as $\delta\rightarrow 0$\\
\item[(b)] By mollifying $g_{(\delta)}$ we construct a family of smooth metrics which approximate $g$ in the $C^0$-sense and have bi-curvature at least
$\kappa -\varepsilon(\delta)$.
\end{itemize}
As mentioned above, (a) holds iff
\begin{eqnarray}
\mc R_\delta(\alpha_\delta,\alpha_\delta)+\mc R_\delta(\beta_\delta,\beta_\delta)
\geq
\kappa-\varepsilon(\delta)
\end{eqnarray}
for all $\alpha_\delta$, $\beta_\delta$ satisfying $\|\alpha_\delta\|_\delta$, $\|\beta_\delta\|_\delta=1$ and 
$\la\alpha_\delta,\beta_\delta\ra_\delta=0$ (where $\la\cdot,\cdot\ra_\delta=\mc I(g_\delta)$). In what follows we will call such 2-forms $g_\delta$-orthonormal.
Theorem \ref{operatorzerlegung} implies
\begin{eqnarray*}
&&\mc R_\delta(\alpha_\delta,\alpha_\delta)+\mc R_\delta(\beta_\delta,\beta_\delta)\\
&=&
\mc R(\alpha_\delta,\alpha_\delta)+\mc R(\beta_\delta,\beta_\delta) 
- f^2_\delta\bigl( \mc A(\alpha_\delta,\alpha_\delta)+\mc A(\beta_\delta,\beta_\delta)\bigr) 
+ f_\delta \bigl( \mc B(\alpha_\delta,\alpha_\delta)+\mc B(\beta_\delta,\beta_\delta)\bigr)\\
&-& 2 f'_\delta\bigl( \mc L(\alpha_\delta,\alpha_\delta)+\mc L(\beta_\delta,\beta_\delta)\bigr)
                            +2 f^2_\delta \bigl( \mc L^2(\alpha_\delta,\alpha_\delta)+\mc L^2(\beta_\delta,\beta_\delta)\bigr)
                            +2 Cf_\delta \bigl( \hat{\mc I}(\alpha_\delta,\alpha_\delta)+\hat{\mc I}(\beta_\delta,\beta_\delta)\bigr)\\
&+&\bigl( \mc E(\delta)(\alpha_\delta,\alpha_\delta)+\mc E(\delta)(\beta_\delta,\beta_\delta)\bigr)
\end{eqnarray*}
where $\mc E(\delta)$ is an operator whose eigenvalues tend to zero uniformly on $M_0$. 
Since $g_\delta\rightarrow g_0$ uniformly on $M_0$, for small enough $\delta$ any 
$g_\delta$-orthonormal forms $\alpha_\delta$ and $\beta_\delta$
are uniformly bounded with respect to $g_0$ by some fixed constant. Thus, we can estimate the $\mc E(\delta)$ terms  from below by $-\ve(\delta)$.
$\mc L$ is positive semidefinite and bounded near $\Gamma$, and $f'_\delta$ does not exceed $\delta$. Therefore,  
$- 2 f'_\delta\bigl( \mc L(\alpha_\delta,\alpha_\delta)+\mc L(\beta_\delta,\beta_\delta)\bigr)\geq-\ve(\delta)$.
Finally, the $\mc L^2$ terms are nonnegative and we arrive at
\begin{eqnarray}\label{xxx}
&&\mc R_\delta(\alpha_\delta,\alpha_\delta)+\mc R_\delta(\beta_\delta,\beta_\delta)\nonumber\\
&\geq&
\mc R(\alpha_\delta,\alpha_\delta)+\mc R(\beta_\delta,\beta_\delta) 
- f^2_\delta\bigl( \mc A(\alpha_\delta,\alpha_\delta)+\mc A(\beta_\delta,\beta_\delta)\bigr) 
+ f_\delta \bigl( \mc B(\alpha_\delta,\alpha_\delta)+\mc B(\beta_\delta,\beta_\delta)\bigr)\\
&+& Cf_\delta \bigl( \hat{\mc I}(\alpha_\delta,\alpha_\delta)+\hat{\mc I}(\beta_\delta,\beta_\delta)\bigr)
-\varepsilon(\delta)\nonumber
\end{eqnarray}
By applying the Gram-Schmidt process to $\alpha_\delta$ and $\beta_\delta$ and putting 
\[\tilde\alpha_\delta:=\frac{\alpha_\delta}{\|\alpha_\delta\|_0}\]
and
\[
\tilde\beta_\delta:=\frac{\beta_\delta-\la\tilde\alpha_\delta,\beta_\delta\ra_0\tilde\alpha_\delta}
                           {\|\beta_\delta-\la\tilde\alpha_\delta,\beta_\delta\ra_0\tilde\alpha_\delta\|_0}
\]
we obtain $g_0$-orthonormal 2-forms $\tilde\alpha_\delta$, $\tilde\beta_\delta$ satisfying
\[\|\tilde\alpha_\delta-\alpha_\delta\|_0,\:\|\tilde\beta_\delta-\beta_\delta\|_0\leq\varepsilon(\delta)\]
independent of the initial $\alpha_\delta$, $\beta_\delta$. Since $f_\delta$ and all of the operators on the 
right hand side of \eqref{xxx} are uniformly bounded near $\Gamma$ we may replace 
$\alpha_\delta$, $\beta_\delta$ by $\tilde\alpha_\delta$, $\tilde\beta_\delta$ and the inequality will still hold up to an $-\varepsilon(\delta)$:
\begin{eqnarray}
&&\mc R_\delta(\alpha_\delta,\alpha_\delta)+\mc R_\delta(\beta_\delta,\beta_\delta)\nonumber\\
&\geq&
\mc R(\tilde\alpha_\delta,\tilde\alpha_\delta)+\mc R(\tilde\beta_\delta,\tilde\beta_\delta) 
- f^2_\delta\bigl( \mc A(\tilde\alpha_\delta,\tilde\alpha_\delta)+\mc A(\tilde\beta_\delta,\tilde\beta_\delta)\bigr) 
+ f_\delta \bigl( \mc B(\tilde\alpha_\delta,\tilde\alpha_\delta)+\mc B(\tilde\beta_\delta,\tilde\beta_\delta)\bigr)\nonumber\\
&+& 2Cf_\delta \bigl( \hat{\mc I}(\tilde\alpha_\delta,\tilde\alpha_\delta)+\hat{\mc I}(\tilde\beta_\delta,\tilde\beta_\delta)\bigr)
-\varepsilon(\delta)\nonumber
\end{eqnarray}
From construction $\tilde\alpha_\delta$ and $\tilde\beta_\delta$ are $g_0$-orthonormal on $M_0$ and $g_1$-orthonormal on $\Gamma$
(recall that $g_0=g_1$ on $\Gamma$). By adopting the argument from Lemma \ref{lemma92} we arrive at
\begin{eqnarray}
&&\mc R_\delta(\alpha_\delta,\alpha_\delta)+\mc R_\delta(\beta_\delta,\beta_\delta)\nonumber\\
&\geq&
\kappa
+2 f_\delta
\bigl[
\bigl(-\mc L^2+\frac 12 \nabla^2_N \mc G_1+ C\hat{\mc I} \bigr)(\tilde\alpha_\delta,\tilde\alpha_\delta)
+\bigl(-\mc L^2+\frac 12 \nabla^2_N \mc G_1+ C\hat{\mc I} \bigr)(\tilde\beta_\delta,\tilde\beta_\delta)
\bigr]\nonumber\\
&-&\varepsilon(\delta)\nonumber
\end{eqnarray}
Since $-\mc L^2+\frac 12 \nabla^2_N \mc G_1+ C\hat{\mc I}$
is positive semidefinite for large enough fixed $C$ and uniformly bounded near $\Gamma$ (cf. proof of Lemma \ref{finest}), (a) follows.
\\
\\
(b) Mollifying $g_{(\delta)}$:\\
\\
Let us fix a $\delta>0$ and define the mollified metric $g_{(\delta)}^h$ in the same way as in Section 6. Our goal is to show 
\begin{eqnarray}\label{vvv}
\mc R(g_{(\delta)}^h)(\alpha,\alpha)+\mc R(g_{(\delta)}^h)(\beta,\beta)\geq \kappa -\varepsilon(\delta)
\end{eqnarray}
for all $g_{(\delta)}^h$-orthonormal $\alpha$, $\beta$.
The computations in Section 6 were carried out for 2-forms with constant coefficients, which we no longer can assume for orthonormal 2-forms.
\\
\\
Using \eqref{ungl2} we obtain 
\[\mathcal R(g^h_{(\delta)})(\alpha,\alpha)+\mathcal R(g^h_{(\delta)})(\beta,\beta)
\geq
(\mathcal R(g_{(\delta)}))^h(\alpha,\alpha) + (\mathcal R(g_{(\delta)}))^h(\beta,\beta)
  -\tilde\varepsilon(\delta,h)\bigl(\|\alpha\|_{(\delta)}^2+\|\beta\|_{(\delta)}^2\bigr)\]
  where $\tilde\ve(\delta,h)\rightarrow 0$ as $h\rightarrow 0$ for every fixed $\delta$.
Since $\alpha$ and $\beta$ have unit length with respect to $g_{(\delta)}^h$ and $g_{(\delta)}^h\stackrel{h\rightarrow 0}{\rightarrow} g_{(\delta)}$,
we can estimate the last term on the right hand side from below by $-\varepsilon (\delta)$ for small enough $h$.
Thus, \eqref{vvv} follows if we show
\begin{eqnarray}\label{bbb}
(\mathcal R(g_{(\delta)}))^h(\alpha,\alpha) + (\mathcal R(g_{(\delta)}))^h(\beta,\beta)\geq \kappa-\tilde\varepsilon(\delta)
\end{eqnarray}
for small enough $h$ and all $g_{(\delta)}^h$-orthonormal $\alpha$, $\beta$.
\\
\\
Let us now fix a point $x\in M$ and some $g_{(\delta)}^h(x)$-orthonormal $\alpha$, $\beta\in\Lambda^2({T_xM})$. 
Recall that in Section 6 we mollified $g_{(\delta)}$ and $R(g_{(\delta)})$ only on a small neighborhood of $\Gamma$ which was 
covered by finitely many coordinate charts $U_1,\dots,U_N$. Off this neighborhood $g$ coincides with $g_{(\delta)}$ and we have
\[(\mathcal R(g_{(\delta)}))^h(\alpha,\alpha) + (\mathcal R(g_{(\delta)}))^h(\beta,\beta)
=\mathcal R(g)(\alpha,\alpha) + \mathcal R(g)(\beta,\beta)\geq \kappa\]
by assumption. Thus w.l.o.g. we can assume that $x\notin\bigcup_{s>N}U_s$. For such $x$ we have
\begin{eqnarray}\label{sss}
&&(\mc R(g_{(\delta)}))^h(x)(\alpha,\alpha)+(\mc R(g_{(\delta)}))^h(x)(\beta,\beta)\\
&=&\sum_{s=1}^N\eta_s(x)\int_{|z|\leq1} \rho(z) (R(g_{(\delta)}))^s_{ijkl}(x-hz)(\alpha_s^{ij}\alpha_s^{kl}+\beta_s^{ij}\beta_s^{kl})dz\nonumber
\end{eqnarray}
where the coefficients refer to the charts $(U_s',\varphi_s)$.
Next we extend $\alpha$, $\beta$ to $U_s'$ in such a way that the extensions are $g_{(\delta)}^h$-orthonormal:
We define 2-forms $\alpha_s$, $\beta_s$ on $U_s'$, $s=1,\dots,N$ by putting
$\alpha_s^{ij}(y):=\alpha_s^{ij}$ and $\beta_s^{ij}(y):=\beta_s^{ij}$ (here we only have to consider the neighborhoods with $x\in U_s$).
Using the Gram-Schmidt process we obtain $g_{(\delta)}^h$-orthonormal 2-forms
\[\tilde\alpha_s=\frac{\alpha_s}{\|\alpha_s\|_{g_{(\delta)}^h}}\]
and
\[
\tilde\beta_s=\frac{\beta_s-\la\tilde\alpha_s,\beta_s\ra_{g_{(\delta)}^h}\tilde\alpha_s}
                           {\|\beta_s-\la\tilde\alpha_s,\beta_s\ra_{g_{(\delta)}^h}\tilde\alpha_s\|_{g_{(\delta)}^h}}
\]
(Note that these extentions might differ on $U_s\setminus\{x\}$.)
By the mean value theorem the right hand side of \eqref{sss} equals to
\begin{eqnarray}\label{www}
&&\sum_{s=1}^N\eta_s(x)\int_{|z|\leq1} \rho(z) (R(g_{(\delta)}))^s_{ijkl}(x-hz)
\bigl(\tilde\alpha_s^{ij}(x-hz)\tilde\alpha_s^{kl}(x-hz)+\tilde\beta_s^{ij}(x-hz)\tilde\beta_s^{kl}(x-hz)\bigr)dz\nonumber\\
&+&h\sum_{s=1}^N\eta_s(x)\int_{|z|\leq1} \rho(z) (R(g_{(\delta)}))^s_{ijkl}(x-hz)
D\bigl( \tilde\alpha_s^{ij}\tilde\alpha_s^{kl}+\tilde\beta_s^{ij}\tilde\beta_s^{kl}\bigr)(\xi^s_{x,hz})z\,dz
\end{eqnarray} 
where $\xi^s_{x,hz}=(1-t)x+thz$ for some $t\in[0,1]$.
Now we apply the Gram-Schmidt process  with respect to $g_{(\delta)}$ to the 2-forms $\tilde \alpha_s$ and $\tilde \beta_s$, and construct
$g_{(\delta)}$-orthonormal $\tilde{\tilde\alpha}_s,\tilde{\tilde\beta}_s$.
The first sum in \eqref{www} is estimated from below by
\begin{eqnarray}\label{zzz}
&&\sum_{s=1}^N\eta_s(x)\int_{|z|\leq1} \rho(z) (R(g_{(\delta)}))^s_{jikl}(x-hz)
\bigl(\tilde{\tilde\alpha}_s^{ij}(x-hz)\tilde{\tilde\alpha}_s^{kl}(x-hz)+\tilde{\tilde\beta}_s^{ij}(x-hz)\tilde{\tilde\beta}_s^{kl}(x-hz)\bigr)\nonumber\\
&-&\varepsilon(\delta,h)
\end{eqnarray}
%where $\tilde{\tilde\alpha}_s$ and $\tilde{\tilde\beta}_s$ are $g_{(\delta)}$-orthonormal and
where
\[
\varepsilon(\delta,h)
\leq c(n)\|(R(g_{(\delta)}))^s_{ijkl}\|_{L^\infty(U_s')}\|(g_{(\delta)}-g_{(\delta)}^h)^s_{ij}\|_{C^0(U_s')}\stackrel{h\rightarrow 0}{\rightarrow} 0
\]
for every fixed $\delta$. Moreover, in view of (a) the integrand in \eqref{zzz} is bounded from below by $\kappa-\varepsilon(\delta)$.
Finally, the second integrand in \eqref{www} is bounded by 
\[
c(n)\|(R(g_{(\delta)}))^s_{ijkl}\|_{L^\infty(U_s')}\|(g_{(\delta)}-g_{(\delta)}^h)^s_{ij}\|_{C^1(U_s')}
\]
and thus
the second expression in \eqref{www} tends to zero uniformly as $h\rightarrow 0$. For small enough $h$ \eqref{bbb} follows with 
$\tilde\varepsilon(\delta)=2\varepsilon(\delta)$ and we are done.
\end{proof}
\end{thm}

\subsection{Manifolds with isotropic curvature $\geq \kappa$}

Given a smooth Riemannian manifold $(M,g)$ we consider the complexification of its tangent bundle $\mathbb C \otimes_{\mathbb R} TM$
and the complex-linear extensions of the inner product $g$ and the Riemannian curvature tensor $R$. A complex isotropic two-plane
is spanned by two vectors $Z=X+iY$ and $W=U+iV$ where $X,Y,U,V\in TM$ are orthonormal with respect to $g$.
The isotropic curvature of such a two-plane $P$ is defined as
\[
K(P)=R(Z,W,\bar Z,\bar W)
\]
Using the Bianchi identity \[R(X,Y,U,V)+R(X,V,Y,U)+R(X,U,V,Y)=0\] one easily verifies
\[
K(P)=R(X,U,X,U)+R(X,V,X,V)+R(Y,U,Y,U)+R(Y,V,Y,V)-2R(X,Y,U,V)
\]
Given an isotropic two-plane $P$ spanned by $X+iY$ and $U+iV$ one computes using the Bianchi identity once more
\begin{eqnarray}\label{iso}
K(P)=\mc R(\alpha,\alpha)+\mc R(\beta,\beta)
\end{eqnarray}
where $\alpha=X\wedge U+V\wedge Y$ and $\beta=X\wedge V+Y\wedge U$.
We say that a Riemannian manifold has isotropic curvature $\geq \kappa$ if $K(P)\geq \kappa$ holds for all isotropic two-planes of $M$. 
$M$ has 1-isotropic (2-isotropic) curvature $\geq \kappa$ if $M\times \mathbb R$ ($M\times \mathbb R^2$) has isotropic curvature $\geq \kappa$.

\begin{thm}
Let $(M_0,g_0)$, $(M_1,g_1)$, $(M,g)$ and $L=L_0+L_1$ be as in Theorem \ref{main}. 
Suppose that the isotropic (1-isotropic, 2-isotropic) curvatures of $g_0$ and $g_1$ are at least $\kappa$. 
Then the isotropic (1-isotropic, 2-isotropic) curvatures of $g$ is at least $\kappa$ (in a similar sense as in Definition \ref{c0curv}) if
$L$ is positive semidefinite.

\begin{proof}
In view of \eqref{iso}, the proof for the isotropic case is similar as in the previous section.
For the 1-isotropic case let us examine the manifold resulting from gluing $M_1\times \mathbb R$ and 
$M_2\times \mathbb R$ along their boundaries. The boundary of $M_i$, $i=1,2$ is given by
$\Gamma_i\times \mathbb R$. If $\phi:\Gamma_1\rightarrow \Gamma_2$ is  some isometry
of $\Gamma_1$, $\Gamma_2$ with respect to $g_1$, $g_2$, then 
\begin{eqnarray*}
\tilde \phi:\Gamma_1\times \mathbb R&\rightarrow & \Gamma_2\times \mathbb R\\
(x,s)&\mapsto& (\phi(x),s)
\end{eqnarray*}
is an isometry of $\Gamma_1\times \mathbb R$ and $\Gamma_2\times \mathbb R$ with respect to
$g_0\oplus dr$, $g_1\oplus dr$, where $dr$ denotes the standard metric on $\mathbb R$. One easily verifies that 
\[(M_1\times \mathbb R)\cup_{\tilde\phi} (M_2\times \mathbb R)=(M_1\cup_\phi M_2)\times \mathbb R\]
and
\[(g\oplus dr)|_{M_i\times \mathbb R}=g|_{M_i}\oplus dr\]

The inward normal on $\Gamma_i\times \mathbb R$ with respect to $g_i\oplus dr $ is given by $(N,0)$,
where $N$ is the inward normal on $\Gamma_i$ with respect to $g_i$. The second fundamental forms of
$\Gamma_i\times \mathbb R$ are $L_i\oplus 0$, and therefore their sum is positive semidefinite.
We repeat the constructions from Section 2 and define the modified metric $(g_0\oplus dr)_\delta=g_\delta\oplus dr$ on $M_0\times \mathbb R$.
Though $\Gamma\times\mathbb R$ fails to be compact, we may nevertheless proceed as in the isotropic case,
since any operator $\mc T$ which occurs for $M_i\times \mathbb R$ during the proof satisfies $\mc T(x,s)=\mc T(x,0)$, and therefore is
bounded near $\Gamma\times \mathbb R$ due to the compactness of $\Gamma$. The desired smooth metric on
$M\times\mathbb R$, which approximates $g\oplus dr$ and has isotropic curvature $\geq\kappa-\eps(\delta)$ is then given by $g_{(\delta)}\oplus dr$.
The proof for the 2-isotropic case is similar.

\end{proof}

\end{thm}

\subsection{Manifolds with flag curvature $\geq \kappa$}
The flag curvature of an orthonormal three-frame $\{e_1,e_2,e_3\}$ is defined as
\[\fl(e_1,e_2,e_3)=\mc R(e_1\we e_3,e_1\we e_3) +\mc R(e_2\we e_3,e_2\we e_3)\]
\begin{thm}
Let $(M_0,g_0)$, $(M_1,g_1)$, $(M,g)$ and $L=L_0+L_1$ be as in Theorem \ref{main}. 
Suppose that the flag curvatures of $g_0$ and $g_1$ are at least $\kappa$. 
Then the flag curvatures of $g$ is at least $\kappa$ (in a similar sense as in Definition \ref{c0curv}) if
$L$ is positive semidefinite.
\end{thm}
\noindent
The proof is similar as in the bi-curvature case.

\end{document}